\documentclass[11pt]{article}

\usepackage{amsmath,amsfonts,amsthm,amscd,amssymb,graphicx,mathtools}
\usepackage{cite}
\usepackage{bbm}
\usepackage{color}
\usepackage[english=american]{csquotes}
\usepackage[colorlinks=true, pdfstartview=FitV, linkcolor=black, citecolor=black, urlcolor=black]{hyperref}
\usepackage{calc}
\usepackage{mathptmx}
\usepackage{t1enc}
\usepackage{stackrel}
\usepackage{dsfont}

\newtheorem{theo}{Theorem}[section]
\newtheorem{lem}[theo]{Lemma}
\newtheorem{prop}[theo]{Proposition}

\newtheorem{cor}[theo]{Corollary}

\newtheorem{rem}[theo]{Remark}
\newtheorem{defi}[theo]{Definition}

\newcommand{\T}{\mathbb{T}}

\newcommand{\N}{\mathbb{N}}
\newcommand{\R}{\mathbb{R}}

\newcommand{\dxdt}{\mathrm{d}x \mathrm{d}t}
\newcommand{\dx}{ \mathrm{d}x}

\newcommand{\dy}{\mathrm{d}y}
\newcommand{\dd}{\mathrm{d}}

\renewcommand{\epsilon}{\varepsilon}
\renewcommand{\phi}{\varphi}

\numberwithin{equation}{section}
\DeclareMathAlphabet{\mathcal}{OMS}{cmsy}{m}{n}

\begin{document}
\allowdisplaybreaks
	%% TITLE MATTERS
	
	\title{On the Selection of Measure-Valued Solutions for the Isentropic Euler System}
	
	\author{Dennis Gallenm\"uller\footnotemark[1] \and Emil Wiedemann\footnotemark[1]}
	
	\date{}
	
	\maketitle
	
	\begin{abstract}
		
		Measure-valued solutions to fluid equations arise naturally, for instance as vanishing viscosity limits, yet exhibit non-uniqueness to a vast extent. In this paper, we show that some measure-valued solutions to the two-dimensional isentropic compressible Euler equations, although they are energy admissible, can be discarded as unphysical, as they do not arise as vanishing viscosity limits. In fact, these measure-valued solutions also do not arise from a sequence of weak solutions of the Euler equations, in contrast to the incompressible case. Such a phenomenon has already been observed by Chiodaroli, Feireisl, Kreml, and Wiedemann using an $\mathcal{A}$-free rigidity argument, but only for non-deterministic initial datum. We develop their rigidity result to the case of non-constant states and combine this with a compression wave solution evolving into infinitely many weak solutions, constructed by Chiodaroli, De Lellis, and Kreml. Hereby, we show that there exist infinitely many generalized measure-valued solutions to the two-dimensional isentropic Euler system with quadratic pressure law, which behave deterministically up to a certain time and which cannot be generated by weak solutions with bounded energy or by vanishing viscosity sequences.

		%\noindent\textsc{}%MSC (2010): 35Q31 (primary); 35Q30, 35L65, 76N10.}
		%\noindent\textsc{)%Keywords: }
		
	\end{abstract}
	
	\renewcommand{\thefootnote}{\fnsymbol{footnote}}
	%\footnotetext[1]{Institute of Applied Mathematics and Mechanics, University of Warsaw, Banacha 2, 02-097 Warszawa, Poland. Email: t.debiec@mimuw.edu.pl}
	
	\footnotetext[1]{Institut f\"ur Angewandte Analysis, Universit\"at Ulm, Helmholtzstra\ss e 18, 89081 Ulm, Germany. Email:
		emil.wiedemann@uni-ulm.de, dennis.gallenmueller@uni-ulm.de}

	\section{Introduction}
	
	Almost three centuries ago, Euler formulated the system of equations
	\begin{align}
		\begin{split}
			\partial_t\rho+\operatorname{div}_x(\rho u)&=0,\\
			\partial_t(\rho u)+\operatorname{div}_x(\rho u\otimes u)+\nabla_x p(\rho)&=0\label{eq:CEoriginal}
		\end{split}
	\end{align}
	with unknown density $\rho\geq 0$ and velocity $u$. For the corresponding Cauchy problem on $[0,T]\times \Omega$, say $\Omega$ is a smooth and bounded domain $\Omega\subset \R^d$, the additional equation $(\rho,u)(0,\cdot)=(\rho_0,u_0)$ is required to hold in an appropriate sense for some given $\rho_0,u_0$. These equations describe the conservation of mass and momentum of an isentropic compressible perfect fluid. The pressure $p(\rho)\geq 0$ is a pre-determined continuous function of the density and shall satisfy $p(0)=0$ and $p'(\rho)>0$. In many situations, the constitutive relation for the pressure is $p(\rho)=\rho^{\gamma}$ for some $\gamma>1$,\ e.g.\ if a perfect gas is studied. In the present paper, the case $\gamma=2$ will be of particular interest.\\
	In modern PDE theory the notion of weak (or distributional) solutions plays a central role, as actually observed fluid flows often possess low regularity. Especially for the compressible Euler system, weak solutions are interesting as they can be used to model the non-smooth behaviour of turbulence or to describe shock solutions which arise even classically. However, it has recently been discovered that weak solutions might not be unique even upon imposition of an entropy condition, for example cf.~\cite{DS10} and~\cite{Chi14}.\\
	Originally introduced by DiPerna and Majda~\cite{DM87} for the incompressible Euler system and further developed by Neustupa~\cite{Neu93} to the compressible Euler system, there is the even weaker notion of measure-valued solutions. This concept describes parametrised probability measures, also called Young measures, on the phase space $\R^+\times \R^d$ that solve the Euler equations in an average sense, see Definition~\ref{Def:MVS} and~\ref{Def:gMVS} below for details. It is important to note that it is not part of the notion of a measure-valued solution to be somehow obtained from a sequence of approximations (it is precisely the aim of the current contribution to show that not every measure-valued solution has this property). Weakly converging sequences of weak solutions might exhibit oscillations or concentrations in the nonlinear terms. While it suffices to deal with classical Young measures as long as concentration effects can be excluded (e.g.\ when the sequence is uniformly bounded), for fluid equations one usually requires generalized Young measures. We will be concerned with both notions throughout the present paper. 
	\\
	Although weak solutions already fail to be unique in general, the even weaker concept of measure-valued solutions has raised some recent interest, since quite surprising properties have been shown:
	\begin{itemize}
		\item For the incompressible and the compressible Euler system measure-valued solutions enjoy the weak-strong uniqueness property, cf.~\cite{BDS} and~\cite{GSW15}.
		\item There is numerical evidence that for vortex sheet initial data the numerical solution scheme does not converge to the stationary solution, cf.~\cite{FKMT}. This can be interpreted to confirm the usefulness of measure-valued solutions, since it is easier and more realistic to show convergence of numerical schemes to measure-valued solutions instead of weak solutions.
		\item In the incompressible regime every measure-valued solution can be generated by a sequence of weak solutions. This result by~\cite{SW12} leads to the conclusion that for the incompressible Euler system the notions of weak solutions and measure-valued solutions essentially coincide.
	\end{itemize}
	The latter property stands in marked contrast to the compressible Euler system. This has been observed by~\cite{CFKW}, where a measure-valued solution was constructed which cannot be generated by sequences of weak solutions. For that, a rigidity result in the spirit of compensated compactness for so-called $\mathcal{A}$-free sequences generalizing a well-known result of Ball and James, cf.~\cite{BJ87}, was used. The construction of~\cite{CFKW} is in such a way that the measure-valued solution is constant and consists of two separated Dirac measures. In particular, the solution from~\cite{CFKW} does not arise from deterministic initial data.\\
	For us, the question arose if this non-generable nature is only due to the non-deterministic nature of the data, or if this kind of solution could appear after\ e.g.\ a classical shock formation. Moreover, it remained unclear in~\cite{CFKW} whether every measure-valued solution is a vanishing viscosity limit. So we aim to show in the present paper that there exists also a measure-valued solution to the isentropic Euler system that coincides with a classical compression wave solution on a non-empty time interval and evolves afterwards such that it cannot be generated by sequences of weak solutions or by a vanishing viscosity sequence, see Theorem~\ref{Thm:generalizedYM}. As a consequence of our results, cf.\ Corollary~\ref{Cor:selectioncriterion}, one can infer a selection criterion for unphysical solutions: Out of the possibly infinitely many measure-valued solutions corresponding to fixed initial data we can discard those consisting of a convex combination of two Dirac measures supported at weak solutions whose lifted states are not wave-cone-connected. While this is a rather technical condition, it does show that there exist admissible measure-valued solutions that do not arise as viscosity limits and may hence be deemed non-physical (cf.\ e.g.\ the discussion in~\cite{BTW}). Therefore, although there still may remain many solutions, at least some of them can be removed. We hasten to add that all measure-valued solutions considered here satisfy the standard entropy condition, and that the specific measure-valued solution from Theorem~\ref{Thm:generalizedYM} does not only not arise from any viscosity sequence with fixed initial data, but from \emph{any} vanishing viscosity sequence.\\
	For our proof we need to generalize the rigidity result Theorem~2 in~\cite{CFKW} further to hold for Young measures with support on the line segment between non-constant states, see Theorem~\ref{Thm:rigidity} below. These states must not be wave-cone-connected, i.e.\ their difference does not lie in the wave-cone of $\mathcal{A}$, which corresponds to the fact that in the classical rigidity result for differential inclusions of Ball and James, Proposition~2 in~\cite{BJ87}, the difference of the considered matrices must have rank greater than one. Moreover, our generalization to the case of $L^1$-integrable generating sequences is expedient to include in our main result the case of weak solutions with only the integrability $L^{\gamma}\times L^2$ of the energy space. For the generalization to $L^1$, equi-integrability is an important assumption on the generating sequence, which will be guaranteed by the notion of generating Young measures in the generalized sense and the assumption of bounded energy for our generating sequence. We give a self-contained proof for the $L^1$-case and hope that our rigidity result Theorem~\ref{Thm:rigidity} might be found interesting in its own right.\\
	In the past ten years, the method of convex integration has led to major breakthroughs in the field of PDEs and in particular in fluid mechanics. Originally introduced by M. Gromov in a more geometric setting, De Lellis and Sz\'ekelyhidi adapted this concept to the Euler system, cf.~\cite{DS09} and~\cite{DS10}. This method provides a powerful tool to generate infinitely many weak solutions with certain properties as\ e.g.\ $L^{\infty}$-boundedness or prescribed energy profiles. In particular, by means of convex integration it was shown that for the incompressible Euler system existence and non-uniqueness hold for any initial data and any given energy profile, cf.~\cite{Wie11} and~\cite{DS10}. These methods have been transferred to the compressible case by considering constant-in-time densities. In particular, to any smooth initial density and suitably chosen initial velocity there correspond infinitely many weak solutions, as shown in~\cite{Chi14}. It even happens that after the shock formation of a classical Lipschitz compression wave infinitely many weak solutions can arise. The latter phenomenon is shown in~\cite{CDK} and will be utilized in the proof of our Theorem~\ref{Thm:generalizedYM} below.\\
	In Section~\ref{Sect:preliminaries} the notation and general concepts concerning homogeneous differential operators and Young measures used throughout this paper will be established. We will prove in Section~\ref{Sect:rigidity} our rigidity result for non-constant states and equi-integrable generating sequences, cf.\ Theorem~\ref{Thm:rigidity}. For that, a localization procedure of Fonseca and M\"uller~\cite{FM} is used. In Section~\ref{Sect:application} we first use the compression wave solution of~\cite{CDK} and show in Proposition~\ref{Prop:Chiodaroli} that it generates also infinitely many non-wave-cone-connected weak solutions. Finally, with these solutions the desired non-generable measure-valued solutions are then constructed in the main result, Theorem~\ref{Thm:generalizedYM}.
	
	\section{Preliminaries}\label{Sect:preliminaries}
	We consider a linear homogeneous differential operator $\mathcal{A}$ of order one with constant coefficients on $\R^d$,\ i.e.\ $\mathcal{A}$ is of the form
	\begin{align*}
		\mathcal{A}=\sum_{l=1}^{d}A^{l}\partial_l,
	\end{align*}
	where for some $k\in \N$ the matrices $A^{l}\in\R^{k\times m}$ are constant coefficients. Hence, it acts on a function $z\colon\Omega\rightarrow \R^m$ by
	\begin{align*}
		\mathcal{A}z&=\sum_{l=1}^{d}A^{l}\partial_{l}z=\left(\sum_{l=1}^{d}\partial_l\sum_{j=1}^{m}A^{l}_{ij}z_j\right)_{i=1,...,k}=:\left(\sum_{l=1}^{d}\partial_lZ^{l}_i\right)_{i=1,...,k}.
	\end{align*}
	Here, we introduced the $k\times d$-matrix-valued function $\left(Z^{l}_i\right)_{il}$ corresponding to $z$. Further, we write $\mathcal{A}^*$ for the adjoint operator defined by
	\begin{align*}
		\mathcal{A}^*=\sum_{l=1}^{d}\left(A^l \right)^T\partial_l.
	\end{align*}
	This implies that for $\varphi\in C_c^{\infty}\left(\Omega,\R^k\right)$ and $\psi\in C_c^{\infty}\left(\Omega,\R^m\right)$ we have
	\begin{align*}
		\int_{\Omega}\mathcal{A}\psi\cdot\varphi \dx=-\int_{\Omega}\psi\cdot\mathcal{A}^*\varphi \dx.
	\end{align*}
	Define the $k\times m$ matrix
	\begin{align*}
		\mathbb{A}(\xi):=\sum_{l=1}^{d}\xi_lA^{l}
	\end{align*}
	for $\xi\in \R^{d}$. We say $\mathcal{A}$ has the constant rank property if there exists $r\in \N$ with $\operatorname{rank}\mathbb{A}(\xi)=r$ for all $\xi\in \mathbb{S}^{d-1}$.\\
	Further, define the \textit{wave cone} of $\mathcal{A}$ by
	\begin{align*}
		\Lambda_{\mathcal{A}}:=\left\{\bar{z}\in \R^m\backslash\{0\}\,:\, \exists\xi\in \R^{d}\backslash\{0\}\textup{ with }\mathcal{A}(h(x\cdot\xi)\bar{z})=0\textup{ in }\mathcal{D}'\left(\Omega,\R^k\right)\,\forall h\in L^1_{\operatorname{loc}}(\R,\R) \right\}.
	\end{align*}
	We give a simple characterization of the wave cone.
	\begin{lem}\label{lemm:charwavecone}
		For linear homogeneous differential operators $\mathcal{A}$ of order one holds:
		\begin{align*}
			\Lambda_{\mathcal{A}}=\left\{\bar{z}\in\R^m\backslash\{0\}\,:\,\left(\bar{Z}^{l}_i\right)_{il}\textup{ does not have full rank} \right\}.
		\end{align*}
	\end{lem}
	\begin{proof}
		This follows by writing out the definitions of $\mathbb{A}$ and $\bar{Z}$.
	\end{proof}	
	Let us now specify the notions concerning Young measures used throughout this paper:\\
	Let $X\subset \R^m$ be a measurable subset. We write $\mathcal{M}^+(X)$ for the space of non-negative Radon measures on $X$. The subset $\mathcal{M}^1(X)\subset \mathcal{M}^+(X)$ denotes the set of probability measures,\ i.e.\ finite Radon measures $\nu$ satisfying $\nu(X)=1$. For $\Omega\subset \R^d$ measurable and $m\in \mathcal{M}^+(\Omega)$, we denote by $L^{\infty}_{\operatorname{w}}\left(\Omega,m,\mathcal{M}^1(X)\right)$ the space of maps $\nu\colon \Omega\to\mathcal{M}^1(X),\,x\mapsto\nu_x$ that are weakly*-measurable with respect to $m$. This means for all functions $f\in C_b(X)$ the map $x\mapsto\int_{X}^{}f(z)\dd\nu_x(z)=:\langle \nu_x,f\rangle$ is $m$-measurable. If $m$ is the Lebesgue-measure, we omit it in the notation.\\
	A map $\nu\in L^{\infty}_{\operatorname{w}}\left(\Omega,\mathcal{M}^1\left(\R^m\right)\right)$ is called a (classical) Young measure. We say that a sequence $(z_n)\subset L^1\left(\Omega,\R^m\right)$ generates the Young measure $\nu$ if
	\begin{align*}
		\lim\limits_{n\rightarrow \infty}\int_{\Omega}^{}\varphi(x)f(z_n(x))\dx=\int_{\Omega}^{}\varphi(x)\langle \nu_x,f\rangle \dx
	\end{align*}
	for all $\varphi\in L^1(\Omega)$ and $f\in C_0(\R^m)$.\\
	For $1\leq p,q<\infty$ define the set
	\begin{align*}
		\mathbb{S}^{l+m-1}_{p,q}:=\left\{(\beta_1,\beta_2)\in\R^{l+m}\,:\,|\beta_1|^{2p}+|\beta_2|^{2q}=1 \right\}.
	\end{align*}
	Then the set $\mathcal{F}_{p,q}$ is defined to be the collection of all Carath\'eodory functions $f\colon \Omega\times \R^{l+m}\to\R$,\ i.e.\ measurable in the first and continuous in the second component, whose $p$-$q$-recession function
	\begin{align*}
		f^{\infty}(x,\beta_1,\beta_2):=\underset{\underset{\underset{s\rightarrow\infty}{\left(\beta_1',\beta_2'\right)\rightarrow\left(\beta_1,\beta_2\right)}}{x'\rightarrow x}}{\lim}\frac{f\left(x',s^q\beta_1',s^p\beta_2'\right)}{s^{pq}}
	\end{align*}
	exists and is continuous on $\bar{\Omega}\times\mathbb{S}^{l+m-1}_{p,q}$. A generalized Young measure $(\nu,m,\nu^{\infty})$ now is a triple consisting of a classical Young measure $\nu\in L^{\infty}_{\operatorname{w}}\left(\Omega,\mathcal{M}^1\left(\R^{l+m}\right)\right)$, a non-negative measure $m\in\mathcal{M}^+(\bar{\Omega})$, and a map $\nu^{\infty}\in L^{\infty}_{\operatorname{w}}\left(\bar{\Omega},m,\mathcal{M}^1\left(\mathbb{S}^{l+m-1}_{p,q}\right)\right)$. We say that a sequence $(z_n,w_n)\subset L^p\left(\Omega,\R^l\right)\times L^q\left(\Omega,\R^m\right)$ generates the generalized Young measure $(\nu,m,\nu^{\infty})$ if
	\begin{align*}
		\int_{\Omega}\varphi(x)f(x,z_n(x),w_n(x))\dx\rightarrow&\int_{\Omega}\varphi(x)\int_{\R^{l+m}}^{}f(x,\lambda_1,\lambda_2)\dd\nu_x(\lambda_1,\lambda_2)\dx\\
		&+\int_{\bar{\Omega}}\varphi(x)\int_{\mathbb{S}^{m+l-1}_{p,q}}^{}f^{\infty}(x,\beta_1,\beta_2)\dd\nu_x^{\infty}(\beta_1,\beta_2)\dd m(x)
	\end{align*}
	for all $\varphi\in C_c(\bar{\Omega})$ and $f\in\mathcal{F}_{p,q}$.\\
	This generalizes the case of a sequence with only a single $L^p$-integrability. In this case set
	\begin{align*}
		\mathbb{S}^{m-1}_p:=\left\{\beta\in\R^m\,:\,|\beta|^{2p}=1 \right\}
	\end{align*}
	and define the set $\mathcal{F}_p$ to be the collection of Carath\'eodory functions $f\colon \Omega\times \R^m\to\R$ whose $p$-recession function
	\begin{align*}
		f_{\infty}(x,\beta):=\underset{\underset{\underset{s\rightarrow\infty}{\beta'\rightarrow\beta}}{x'\rightarrow x}}{\lim}\frac{f\left(x',s^p\beta'\right)}{s^{p}}
	\end{align*}
	exists and is continuous on $\bar{\Omega}\times \mathbb{S}^{m-1}_p$. A sequence $(z_n)\in L^p(\Omega,\R^m)$ then generates the generalized Young measure $(\nu,m,\nu^{\infty})\in L^{\infty}_{\operatorname{w}}\left(\Omega,\mathcal{M}^1\left(\R^m\right)\right)\times \mathcal{M}^+(\bar{\Omega})\times L^{\infty}_{\operatorname{w}}\left(\bar{\Omega},m,\mathcal{M}^1\left(\mathbb{S}^{m-1}_{p}\right)\right)$ if
	\begin{align*}
		\int_{\Omega}\varphi(x)f(x,z_n(x))\dx\rightarrow\int_{\Omega}\varphi(x)\int_{\R^{m}}^{}f(x,\lambda)\dd\nu_x(\lambda)dx+\int_{\bar{\Omega}}\varphi(x)\int_{\mathbb{S}^{m-1}_{p}}^{}f^{\infty}(x,\beta)\dd\nu_x^{\infty}(\beta)\dd m(x)
	\end{align*}
	for all $\varphi\in C_c(\bar{\Omega})$ and $f\in\mathcal{F}_{p}$.\\
	One observes the following preliminary fact.
	\begin{lem}\label{Lemma:generatingmeasures}
		Let $(z_n)\in L^1\left(\Omega,\R^m\right)$ be a sequence generating the generalized Young measure $(\nu,0,\mu)$, where $\mu$ is a dummy variable completing our notation. Then, the sequence $(z_n)$ generates the (classical) Young measure $\nu$.
	\end{lem}
	\begin{proof}
		Let $f\in C_0(\R^m)$ and $\varphi\in L^1(\Omega,\R)$ be arbitrary. As $f$ is bounded, its $1$-recession function is identically zero, thus $f\in \mathcal{F}_{1}$. Moreover, for all $\varepsilon>0$ let $\varphi_{\varepsilon}\in C_c(\Omega,\R)$ such that $\|\varphi_{\varepsilon}-\varphi\|_{L^1(\Omega)}\leq \frac{\varepsilon}{2}$. Then, we obtain
		\begin{align*}
			\left|\int_{\Omega}^{}\varphi(x)f(z_n(x))\dx-\int_{\Omega}\varphi(x)\langle \nu_{x},f\rangle \dx\right|\leq \left|\int_{\Omega}^{}\varphi_{\varepsilon}(x)f(z_n(x))\dx-\int_{\Omega}\varphi_{\varepsilon}(x)\langle \nu_{x},f\rangle \dx\right|+\varepsilon\overset{n\rightarrow\infty}{\rightarrow}\varepsilon.
		\end{align*}
		Hence, the sequence $(z_n)$ indeed generates the Young measure $\nu$.
	\end{proof}

	\section{$\mathcal{A}$-Free Rigidity for the Case of Equi-Integrable Sequences}\label{Sect:rigidity}
	The aim of this section is to generalize the rigidity result for $\mathcal{A}$-free sequences from~\cite{CFKW}, Theorem~2. They showed that for $1<p<\infty$ an $L^p$-bounded $\mathcal{A}$-free sequence cannot generate a Young measure with support on the line segment between two non-wave-cone-connected constant states. We will generalize this to the case of line segments between two non-constant states and furthermore also to the case $p=1$ for which we have to assume also equi-integrability of the generating sequence.\\
	We note that Theorem~\ref{Thm:rigidity} below also generalizes the compensated compactness result Theorem~1.2 in~\cite{DPR} for the case of a linear differential inclusion of order one to the setting of Young measures.\\	
	We proceed step by step. First the extension to the case $p=1$ is treated. Note that we need not assume the constant rank property yet.
	\begin{prop}\label{Prop:pisone}
		Consider the torus $\mathbb{T}^d$. Let $\mathcal{A}$ be a linear homogeneous differential operator of order one satisfying $k\geq d$, and $1\leq p< \infty$. Let $\bar{z}_1,\bar{z}_2\in\R^m$ be two constant states such that $\bar{z}_2-\bar{z}_1\notin \Lambda_{\mathcal{A}}$. Let $z_n\colon \mathbb{T}^d\to \R^m$ be an equi-integrable family of functions with
		\begin{align*}
			\|z_n\|_{L^p\left(\mathbb{T}^d,\R^m\right)}&\leq c<\infty\text{ and }\\
			\mathcal{A}z_n&=0\text{ in }\mathcal{D}'\left(\mathbb{T}^d\right),
		\end{align*}
		which generates a Young measure $\nu_{x}\in \mathcal{M}^1\left(\R^m\right)$ such that
		\begin{align*}
			\operatorname{supp}(\nu_x)\subset\left\{\lambda\bar{z}_1+(1-\lambda)\bar{z}_2\,:\,\lambda\in[0,1] \right\}
		\end{align*}
		for\ a.e.\ $x\in\mathbb{T}^d$.\\
		Then 
		\begin{align*}
			z_n\rightarrow z_{\infty}\text{ for }1\leq p<\infty\text{ in }L^p\left(\mathbb{T}^d\right)
		\end{align*}
		and
		\begin{align*}
			\nu_x=\delta_{z_{\infty}}
		\end{align*}
		for\ a.e.\ $x\in \mathbb{T}^d$ with $z_{\infty}=\bar{\lambda}\bar{z}_1+(1-\bar{\lambda})\bar{z}_2$ a constant function,\ i.e.\ $\bar{\lambda}\in [0,1]$ fixed.
	\end{prop}
	\begin{rem}
		We need the assumption $k\geq d$ in the following proof to find a left inverse of the full rank $k\times d$-matrix $\bar{Z}_1-\bar{Z}_2$. This condition corresponds in the applications of our rigidity results to the case of formally under- or well-determined systems. In particular, the linearized Euler system (\ref{eq:CEsubsolution}) below is underdetermined.
	\end{rem}
	\begin{proof}
		The case $1<p<\infty$ is treated in Theorem~2 in~\cite{CFKW} and easily translates to the case of the torus as domain. Note that in this case equi-integrability is trivially fulfilled for all $L^p$-bounded families of functions.\\
		\\
		We now change the proof of Theorem~2 in~\cite{CFKW} appropriately to show the case $p=1$:\\
		\textbf{Step 1:}
		For the first step we can consider a general Lipschitz domain $\Omega \in \R^d$, since we only work in $L^1$ and $L^1\left(\T^d\right)\simeq L^1\left([-\pi,\pi]^d\right)$.\\
		Let us first prove that $z_n$ is of the form
		\begin{align}
			z_n(x)=e_n(x)+\lambda_n(x)\bar{z}_1+(1-\lambda_n(x))\bar{z}_2\label{eq:formofzn}
		\end{align}
		for $e_n\rightarrow 0$ in $L^1\left(\Omega,\R^d\right)$ and $\lambda_n\in L^{\infty}\left(\Omega,\R^d\right)$ with $0\leq \lambda_n(x)\leq 1$ for\ a.e.\ $x\in\Omega$.\\
		For this, let $\delta>0$ be arbitrary. Then, there exists a function $F_{\delta}\in C^{\infty}\left(\R^d\right)$, $0\leq F_{\delta}\leq 1$ with
		\begin{align*}
			F_{\delta}\equiv&0 \textup{ on }\left\{\bar{z}\in \R^d\,:\,\operatorname{dist}\left(\bar{z},\big\{\lambda\bar{z}_1+(1-\lambda)\bar{z}_2\,:\, \lambda\in[0,1] \big\}\right)<\frac{\delta}{2} \right\},\\
			F_{\delta}\equiv&1 \textup{ on }\left\{\bar{z}\in \R^d\,:\,\operatorname{dist}\left(\bar{z},\big\{\lambda\bar{z}_1+(1-\lambda)\bar{z}_2\,:\, \lambda\in[0,1] \big\}\right)<\delta \right\}^c.
		\end{align*}
		So, $z_n=F_{\delta}(z_n)z_n+(1-F_{\delta}(z_n))z_n$. Since $z_n$ is equi-integrable, so is $(F_{\delta}(z_n)z_n)$. As $K:=\big\{\lambda\bar{z}_1+(1-\lambda)\bar{z}_2\,:\,\lambda\in[0,1] \big\}\subset \R^m$ is a compact set, we infer from Theorem~2.2 in~\cite{FM} that there is a subsequence $(z_{n})$ with
		\begin{align*}
			\operatorname{dist}(z_{n},K)\rightarrow 0 \text{ in measure}.
		\end{align*}
		Thus, there is yet another subsequence $(z_{n})$ such that $\operatorname{dist}(z_n,K)\rightarrow 0$ a.e.. Hence, by the definition of $F_{\delta}$ there holds $F_{\delta}(z_n(x))z_n(x)\rightarrow 0$ a.e., which implies $F_{\delta}(z_n(x))z_n(x)\rightarrow 0$ in measure. So, by Vitali's convergence theorem, we obtain for all fixed $\delta>0$ that
		\begin{align*}
			F_{\delta}(z_n)z_n\rightarrow 0 \text{ in }L^1\left(\Omega,\R^m\right)
		\end{align*}
		as $n\rightarrow \infty$. Moreover, for fixed $\delta>0$, we rewrite
		\begin{align*}
			(1-F_{\delta}(z_n))z_n=&(1-F_{\delta}(z_n))\left(z_n-\lambda^{\delta}_n\bar{z}_1-\left(1-\lambda^{\delta}_n\right)\bar{z}_2\right)-F_{\delta}(z_n)\left(\lambda^{\delta}_n\bar{z}_1+\left(1-\lambda^{\delta}_n\right)\bar{z}_2\right)\\
			&+\left(\lambda^{\delta}_n\bar{z}_1+\left(1-\lambda^{\delta}_n\right)\bar{z}_2\right)
		\end{align*}
		by choosing functions $0\leq \lambda^{\delta}_n(x)\leq 1$ for a.a. $x\in \Omega$ such that
		\begin{align*}
			\left|(1-F_{\delta}(z_n))\left(z_n-\lambda_n^{\delta}\bar{z}_1-\left(1-\lambda_n^{\delta}\right)\bar{z}_2\right)\right|\leq \delta
		\end{align*}
		and
		\begin{align*}
			\left\|F_{\delta}(z_n)\left(\lambda_n^{\delta}\bar{z}_1+\left(1-\lambda_n^{\delta}\right)\bar{z}_2\right)\right\|_{L^1\left(\Omega,\R^m\right)}\overset{n\rightarrow \infty}{\rightarrow}0.
		\end{align*}
		Therefore, set $\delta_k:=\frac{1}{k}\rightarrow 0$, and for all $k\in\N$ choose $n_k$ large enough such that
		\begin{align*}
			\left\|F_{\frac{1}{k}}\left(z_{n_k}\right)z_{n_k}\right\|_{L^1\left(\Omega,\R^m\right)}&\leq \frac{1}{k},\\
			\left\|F_{\frac{1}{k}}\left(z_{n_k}\right)\left(\lambda_{n_k}^{\frac{1}{k}}\bar{z}_1+\left(1-\lambda_{n_k}^{\frac{1}{k}}\right)\bar{z}_2\right)\right\|_{L^1\left(\Omega,\R^m\right)}&\leq \frac{1}{k}.
		\end{align*}
		Then, defining
		\begin{align*}
			e_k:=F_{\frac{1}{k}}\left(z_{n_k}\right)z_{n_k}+\left(1-F_{\frac{1}{k}}\left(z_{n_k}\right)\right)\left(z_n-\lambda_{n_k}^{\frac{1}{k}}\bar{z}_1-\left(1-\lambda_{n_k}^{\frac{1}{k}}\right)\bar{z}_2\right)-F_{\frac{1}{k}}\left(z_{n_k}\right)\left(\lambda_{n_k}^{\frac{1}{k}}\bar{z}_1+\left(1-\lambda_{n_k}^{\frac{1}{k}}\right)\bar{z}_2\right)
		\end{align*}
		yields the identity (\ref{eq:formofzn}) with $\lambda_k\equiv \lambda^{\frac{1}{k}}_{n_k}$. Moreover, we have
		\begin{align*}
			\|e_k\|_{L^1\left(\Omega,\R^m\right)}\overset{k\rightarrow \infty}{\rightarrow}0.
		\end{align*}
		\textbf{Step 2:}\\
		Now as we will use the Calder\'on-Zygmund theorem, we have to restrict ourselves to the case of the torus $\T^d$ or the whole space $\R^d$ as domain. The case of the torus will be of use for the subsequent results, as it is also a bounded set. So, we only give the proof for $\T^d$.\\
		Rewriting the definitions for the capital letter variables gives
		\begin{align}
			0=(\mathcal{A}z)_i=\operatorname{div} Z_i\label{eq:divcharA}
		\end{align}
		in $\mathcal{D}'\left(\mathbb{T}^d\right)$. Hence, if we define $\left(E^{l}_n(x)\right)_i:=\sum_{j=1}^{m}A^{l}_{ij}e^j_n(x)$, we obtain by (\ref{eq:formofzn}) and linearity of $\mathcal{A}$ for all $i=1,...,k$ that
		\begin{align*}
			0=\operatorname{div}(E_n)_i+(\bar{Z}_1-\bar{Z}_2)_i\cdot\nabla\lambda_n
		\end{align*}
		in $\mathcal{D}'\left(\mathbb{T}^d,\R\right)$. Moreover,
		\begin{align*}
			\|E_n\|_{L^1\left(\mathbb{T}^d,\R^{k\times d}\right)}\leq C(l,m,k,d,\mathcal{A})\|e_n\|_{L^1\left(\mathbb{T}^d,\R^m\right)}\overset{n\rightarrow \infty}{\rightarrow} 0.
		\end{align*}
		Consider now the standard mollifier $\eta\in C_c^{\infty}\left(\R^d,\R\right)$ and let $\frac{\pi}{2}>\varepsilon>0$. Define $z_{n,\varepsilon}:=z_n\ast\eta_{\varepsilon}$. Then, $z_{n,\varepsilon}$ is $\mathcal{A}$-free for all $\frac{\pi}{2}>\varepsilon>0$, since $z_n$ is $\mathcal{A}$-free. Hence, as $z_{n,\varepsilon}$ is a smooth function on the torus, $\mathcal{A}z_{n,\varepsilon}=0$ holds pointwise everywhere on $\T^d$.\\
		Therefore,
		\begin{align*}
			(\bar{Z}_1-\bar{Z}_2)\nabla (\lambda_n\ast\eta_{\varepsilon})=-\operatorname{div} (E_n\ast\eta_{\varepsilon}).
		\end{align*}
		Now by Corollary~\ref{lemm:charwavecone} we know that $\bar{Z}_1-\bar{Z}_2$ has full rank and thus using $k\geq d$ the matrix $\bar{Z}_1-\bar{Z}_2$ possesses a left inverse $B\in \R^{d\times k}$. Hence, we obtain
		\begin{align*}
			\nabla (\lambda_n\ast\eta_{\varepsilon})=-B\cdot\operatorname{div} (E_n\ast\eta_{\varepsilon}).
		\end{align*}
		Taking the divergence on both sides yields
		\begin{align*}
			-\Delta(\lambda_n\ast\eta_{\varepsilon})=\operatorname{div}\operatorname{div}(B(E_n\ast\eta_{\varepsilon})).
		\end{align*}
		The H\"ormander multiplier theorem applied to the double Riesz transform $(-\Delta)^{-1}\operatorname{div}\operatorname{div}$ on the torus gives
		\begin{align*}
			\left\|\lambda_n\ast\eta_{\varepsilon}-\frac{1}{|\T^d|}\int_{\T^d}^{}\lambda_n\ast\eta_{\varepsilon}(x)\dx\right\|_{L^{1,\infty}\left(\T^d,\R\right)}\leq C\left\|E_n\ast\eta_{\varepsilon}-\frac{1}{|\T^d|}\int_{\T^d}^{}E_n\ast\eta_{\varepsilon}(x)\dx\right\|_{L^1\left(\T^d,\R^{k\times d}\right)}.
		\end{align*}
		Note that zero averages are needed to apply the H\"ormander multiplier theorem on the torus.\\
		Since $E_n$ lies in $L^1\left(\T^d,\R^{k\times d}\right)$, we obtain
		\begin{align*}
			\|E_n\ast\eta_{\varepsilon}-E_n\|_{L^1\left(\T^d,\R^{k\times d}\right)}\overset{\varepsilon\rightarrow 0}{\rightarrow}0.
		\end{align*}
		Similarly, for $\lambda_n\in L^{\infty}\left(\T^d,\R\right)$ it holds that
		\begin{align*}
			\lambda_n\ast\eta_{\varepsilon}\overset{\varepsilon\rightarrow 0}{\rightarrow} \lambda_n\text{ in }L^1\left(\T^d,\R\right),
		\end{align*}
		which implies
		\begin{align*}
			\lambda_n\ast\eta_{\varepsilon}\overset{\varepsilon\rightarrow 0}{\rightarrow} \lambda_n\text{ in }L^{1,\infty}\left(\T^d,\R\right).
		\end{align*}
		Note that for a function $f\in L^1\left(\T^d,\R\right)$ we have
		\begin{align*}
			\frac{1}{|\T^d|}\int_{\T^d}^{}(f\ast\eta_{\varepsilon})(x)\dx=\frac{1}{|\T^d|}\int_{\T^d}f(y)\dy
		\end{align*}
		and
		\begin{align*}
			\left\|\frac{1}{|\T^d|}\int_{\T^d}f(y)\dy\right\|_{L^1\left(\T^d,\R\right)}=\int_{\T^d}|f|(y)\dy=\|f\|_{L^1\left(\T^d,\R\right)}.
		\end{align*}
		As $\|f-g+g\|_{L^{1,\infty}}\leq 2\|f-g\|_{L^{1,\infty}}+2\|g\|_{L^{1,\infty}}$, we obtain
		\begin{align*}
			&\left\|\lambda_n-\frac{1}{|\T^d|}\int_{\T^d}^{}\lambda_n(x)\dx\right\|_{L^{1,\infty}\left(\T^d,\R\right)}\\
			\leq &\lim\limits_{\varepsilon\rightarrow 0}2\|\lambda_n-\lambda_n\ast\eta_{\varepsilon}+0\|_{L^{1,\infty}\left(\T^d,\R\right)}+\lim\limits_{\varepsilon\rightarrow 0}2\left\|\lambda_n\ast\eta_{\varepsilon}-\frac{1}{|\T^d|}\int_{\T^d}^{}\lambda_n\ast\eta_{\varepsilon}(x)dx\right\|_{L^{1,\infty}\left(\T^d,\R\right)}\\
			\leq &\lim\limits_{\varepsilon\rightarrow 0}2C\left(\|E_n\ast\eta_{\varepsilon}\|_{L^1\left(\T^d,\R^{k\times d}\right)}+\left\|\frac{1}{|\T^d|}\int_{\T^d}^{}E_n\ast\eta_{\varepsilon}(x)dx\right\|_{L^1\left(\T^d,\R^{k\times d}\right)}\right)\\
			=&4C\|E_n\|_{L^1\left(\T^d,\R^{k\times d}\right)}\\
			\overset{n\rightarrow \infty}{\rightarrow}&0.
		\end{align*}
		For all $n\in \N$ we have $\frac{1}{|\T^d|}\int_{\T^d}^{} \lambda_n\dx \leq \frac{1}{|\T^d|}\int_{\T^d}^{}1\dx=1$. Thus, for a subsequence it holds that $\frac{1}{|\T^d|}\int_{\T^d}^{} \lambda_n\dx \overset{n\rightarrow \infty}{\rightarrow} \bar{\lambda}\in [0,1]$. Hence, the corresponding subsequence $(z_n)$ converges in measure to the constant function $\bar{\lambda}\bar{z}_1+(1-\bar{\lambda})\bar{z}_2$. This in turn implies by Vitali's convergence theorem using the equi-integrability of $(z_n)$ that
		\begin{align*}
			z_n\rightarrow\bar{\lambda}\bar{z}_1+(1-\bar{\lambda})\bar{z}_2\text{ in }L^1\left(\T^d,\R^m\right).
		\end{align*}
		In particular, it follows that
		\begin{align*}
			\int_{\T^d}^{}\varphi(x)\langle \nu,g\rangle \dx\leftarrow\int_{\T^d}^{}\varphi(x)g(z_n(x))\dx\rightarrow\int_{\T^d}^{}\varphi(x)g(\bar{\lambda}\bar{z}_1+(1-\bar{\lambda})\bar{z}_2)\dx
		\end{align*}
		for all $g\in C_0\left(\R^m\right)$ and $\varphi\in L^1\left(\T^d\right)$. This implies that $\nu_x=\delta_{\bar{\lambda}\bar{z}_1+(1-\bar{\lambda})\bar{z}_2}$ for\ a.e.\ $x\in \T^d$.
	\end{proof}
	This result can be applied to the case of a general domain $\Omega\subset \R^d$. The proof foreshadows the localization argument, which will also be used in Theorem~\ref{Thm:rigidity} below.
	\begin{cor}\label{Cor:pisoneonOmega}
		Let $\Omega\subset \R^d$ be an open bounded domain, $\mathcal{A}$ a linear homogeneous constant rank differential operator of order one satisfying $k\geq d$, and $1\leq p< \infty$. Let $\bar{z}_1,\bar{z}_2\in\R^m$ be two constant states such that $\bar{z}_2-\bar{z}_1\notin \Lambda_{\mathcal{A}}$. Let also $z_n\colon \Omega\to \R^m$ be an equi-integrable family of functions with
		\begin{align*}
			\|z_n\|_{L^p\left(\Omega,\R^m\right)}&\leq c<\infty\text{ and }\\
			\mathcal{A}z_n&=0\text{ in }\mathcal{D}'\left(\Omega,\R^m\right), 
		\end{align*}
		which generates a Young measure $\nu_{x}\in \mathcal{M}^1\left(\R^m\right)$ such that
		\begin{align*}
			\operatorname{supp}(\nu_x)\subset\left\{\lambda\bar{z}_1+(1-\lambda)\bar{z}_2\,:\,\lambda\in[0,1] \right\}
		\end{align*}
		for\ a.e.\ $x\in\Omega$.\\
		Then 
		\begin{align*}
			z_n\rightarrow z_{\infty}\text{ for }1\leq p<\infty\text{ in }L^p\left(\Omega,\R^m\right)
		\end{align*}
		and
		\begin{align*}
			\nu_x=\delta_{z_{\infty}(x)}
		\end{align*}
		for\ a.e.\ $x\in \Omega$ with $z_{\infty}\in L^p\left(\Omega,\R^m\right)$. In the case $1<p<\infty$ we have that $z_{\infty}=\bar{\lambda}\bar{z}_1+(1-\bar{\lambda})\bar{z}_2$ is a constant function,\ i.e.\ $\bar{\lambda}\in [0,1]$ fixed.
	\end{cor}	
	\begin{rem}
		The constant rank property of $\mathcal{A}$ is used here, as it is required in the localization scheme in Proposition~3.8 from~\cite{FM}.
	\end{rem}
	For convenience we state the aforementioned result from \cite{FM}.
	\begin{prop}[Proposition 3.8 from \cite{FM}]
		Let $1\leq p<+\infty$, and let $\{z_n\}$ be a $p$-equi-integrable sequence in $L^p(\Omega;\R^m)$ such that $\mathcal{A}z_n\rightarrow 0$ in $W^{-1,p}(\Omega)$ if $1<p<\infty$, $\mathcal{A}z_n\rightarrow 0$ in $W^{-1,r}(\Omega)$ for some $r\in (1,d/(d-1))$ if $p=1$, and $ \{z_n\}$ generates the Young measure $\nu=\{\nu_x\}_{x\in\Omega}$. Let $z_n\rightharpoonup z$ in $L^p(\Omega;\R^m)$. Then for a.e. $a\in\Omega$ there exists a sequence $\{\bar{z}_n\}\subset L^p(\T^d;\R^m)\cap \ker\mathcal{A}$ that is $p$-equi-integrable, generates the Young measure $\nu_a$, and satisfies
		\begin{align*}
			\int\limits_{\T^d}^{}\bar{z}_n\dx=\langle \nu_a,\operatorname{id}\rangle=z(a).
		\end{align*}
		In particular, one has
		\begin{align*}
			\langle \nu_a,f\rangle\geq f(\langle \nu_a,\operatorname{id}\rangle)=f(z(a))
		\end{align*}
		for a.e. $a\in\Omega$ and for every continuous $\mathcal{A}$-quasi-convex $f$ that satisfies
		\begin{align*}
			|f(w)|\leq C(1+|w|^p)
		\end{align*}
		for some $C>0$ and all $w\in\R^m$.
	\end{prop}
	\begin{proof}[Proof of Corollary \ref{Cor:pisoneonOmega}]
		The case $1<p<\infty$ is treated in~\cite{CFKW}. Note that the much stronger conclusion of $z_{\infty}$ being constant stems from the fact that we do not need to apply a localization of the generating sequence.\\
		So, we only consider the case $p=1$:\\
		Because the sequence $(z_n)$ is equi-integrable, there exists a subsequence $(z_n)$ and some $z\in L^1\left(\Omega,\R^m\right)$ such that $z_n\rightharpoonup z$ in $L^1\left(\Omega,\R^m\right)$. This sequence $(z_n)$ is therefore equi-integrable, has a weak limit in $L^1$, is $\mathcal{A}$-free, and generates the Young measure $\nu$. Thus, by Proposition~3.8 in~\cite{FM} for\ a.e.\ $a\in \Omega$ there exists an $\mathcal{A}$-free sequence $(\bar{z}_n^a)\in L^1\left(\T^d,\R^m\right)$ generating the homogeneous Young measure $\nu_a$. Since $\operatorname{supp}(\nu_a)\subset \left\{\lambda\bar{z}_1+(1-\lambda)\bar{z}_2\,:\,\lambda\in[0,1] \right\}$, we infer from Proposition~\ref{Prop:pisone} that $\nu_a=\delta_{\lambda_a\bar{z}_1+(1-\lambda_a)\bar{z}_2}$ for some number $\lambda_a\in[0,1]$. Thus, $(z_n)$ generates the Young measure $a\mapsto \nu_a=\delta_{\lambda_a\bar{z}_1+(1-\lambda_a)\bar{z}_2}$. This implies that $z_n$ converges to the function $x\mapsto \delta_{\lambda_x\bar{z}_1+(1-\lambda_x)\bar{z}_2}$ in measure, which by Vitali's convergence theorem and the equi-integrability of $(z_n)$ means that $z_n\rightarrow \left(x\mapsto \lambda_x\bar{z}_1+(1-\lambda_x)\bar{z}_2\right)$ in $L^1(\Omega,\R^m)$.
	\end{proof}
	Now we prove that $\mathcal{A}$-free rigidity holds also for Young measures with support in the line segment between two variable states that are not wave-cone-connected.
	\begin{theo}\label{Thm:rigidity}
		Let $\Omega\subset \R^{d}$ be an open bounded domain, $\mathcal{A}$ a linear homogeneous constant rank differential operator of order one satisfying $k\geq d$, and $1\leq p<\infty$. Further, let $\bar{z}_1,\bar{z}_2\in L^{\infty}\left(\Omega,\R^m\right)$ be such that $\bar{z}_2-\bar{z}_1\notin\Lambda_{\mathcal{A}}$ on a set of positive measure $A\subset\Omega$. Assume $z_n\colon \Omega\mapsto \R^m$ is an equi-integrable family of functions such that
		\begin{align*}
			\|z_n\|_{L^{p}\left(\Omega,\R^m\right)}&\leq c<\infty \text{ and }\\
			\mathcal{A}z_n&\rightarrow 0\text{ in }W^{-1,r}\left(\Omega,\R^m\right)
		\end{align*}
		for some $r\in \left(1,\frac{d}{d-1}\right)$. Further, assume that $(z_n)$ generates a compactly supported Young measure $\nu\in L^{\infty}_w\left(\Omega,\mathcal{M}^1\left(\R^m\right)\right)$ such that
		\begin{align*}
			\operatorname{supp}(\nu_{x})\subset \left\{\lambda \bar{z}_1(x)+(1-\lambda)\bar{z}_2(x)\,:\,\lambda\in [0,1] \right\}
		\end{align*}
		for\ a.e.\ $x\in \Omega$.\\
		Then, for\ a.e.\ $x\in A$ it holds that
		\begin{align*}
			\nu_{x}=\delta_{w(x)}
		\end{align*}
		with $w\in L^1\left(A,\R^m\right)$ and $z_n\rightarrow w$ in $L^1\left(A,\R^m\right)$.
	\end{theo}	
	\begin{proof}
		After choosing a subsequence, one has $z_n\rightharpoonup z$ in $L^p\left(\Omega,\R^m\right)$ for some $z\in L^p\left(\Omega,\R^m\right)$. Proposition~3.8 in~\cite{FM} provides us for\ a.e.\ $a\in A$ with an equi-integrable sequence $(\bar{z}_n^a)\in L^1\left(\T^{d},\R^m\right)$ which is $\mathcal{A}$-free on $\Omega$ and generates the Young measure ${\nu}_a$. Note that $\operatorname{supp}({\nu}_a)\subset \left\{\lambda \bar{z}_1(a)+(1-\lambda)\bar{z}_2(a)\,:\, \lambda\in[0,1] \right\}$. The latter is a compact set, as $\bar{z}_1,\bar{z}_2$ are bounded\ a.e.\ in $\Omega$. Equi-integrability of $(\bar{z}_n^a)$ implies its $L^1$-boundedness. Therefore, for almost every $a\in A$ Proposition~\ref{Prop:pisone} implies that ${\nu}_a=\delta_{z^a_{\infty}}$ for some $z^a_{\infty}\in \left\{\lambda \bar{z}_1(a)+(1-\lambda)\bar{z}_2(a)\,:\, \lambda\in[0,1] \right\}$ a constant. So, define $w\colon A\to\R^m,\,x\mapsto z^{x}_{\infty}$. It follows that $(z_n)$ converges on $A$ to the function $w$ in measure, since $(z_n)$ generates the atomic Young measure $\nu=\delta_w$ on $A$. Finally, Vitali's convergence theorem and the equi-integrability of $(z_n)$ imply that $z_n\rightarrow w$ in $L^1\left(A,\R^m\right)$.
	\end{proof}
		
	\section{Application to the Two-Dimensional Isentropic Euler System}\label{Sect:application}
	We aim to show that there is a generalized measure-valued solution to the isentropic Euler system on $(0,\infty)\times \R^2$ which is not generated by a sequence of weak solutions and has Lipschitz initial data.\\
	For the following definitions let $\Omega\subset \R^2$ be open and $T>0$. First, the common framework of distributional solutions is introduced.
	\begin{defi}
		A pair $(\rho,u)\in L^1\left([0,T]\times \Omega,\R^+\times\R^2\right)$ is a \textit{weak solution} to (\ref{eq:CEoriginal}) with initial data $(\rho_0,u_0)$ if $(\rho_0,\rho_0u_0)\in L^1\left(\Omega,\R^+\times \R^2\right)$ and
		\begin{align*}
			\int_{0}^{T}\int_{\Omega}^{}\partial_t\psi \rho+\nabla_x\psi\cdot\rho u \dxdt + \int_{\Omega}^{}\psi(0,x)\rho_0(x)\dx&=0,\\
			\int_{0}^{T}\int_{\Omega}^{}\partial_t\varphi\cdot \rho u+\nabla_x\varphi:(\rho u\otimes u)+\operatorname{div}_x\varphi p(\rho) \dxdt+\int_{\Omega}^{}\varphi(0,x)\cdot\rho_0(x)u_0(x)\dx&=0
		\end{align*}
		for all $\psi\in C_c^{\infty}([0,\T)\times\Omega)$ and $\varphi\in C_c^{\infty}\left([0,T)\times\Omega,\R^3 \right)$. The above integrals have to exist as part of the definition.\\
		Moreover, we say that $(\rho, u)$ is a \textit{finite energy weak solution} if $(\rho, u)$ is a weak solution and $(\rho,\sqrt{\rho}u)\in L^{\gamma}\left([0,T]\times \Omega,\R^+\right)\times L^2\left([0,T]\times \Omega,\R^2\right)$.
	\end{defi}
	In particular, a weak solution is a finite energy weak solution in the sense of the definition above, if its energy $E_{(\rho,u)}(t)$ is uniformly bounded in time. Here, the \textit{energy} of a weak solution $(\rho,u)$ to (\ref{eq:CEoriginal}) is defined as
	\begin{align*}
		E_{(\rho,u)}(t):=\frac{1}{2}\int_{\Omega}\rho(t,x)|u(t,x)|^2dx+\frac{1}{\gamma-1}\int_{\Omega}\rho^{\gamma}(t,x)dx
	\end{align*}
	for $t\in[0,T]$.\\
	We now specify the notion of measure-valued solutions, following~\cite{CFKW}. The first concept concerns only oscillations. 
	\begin{defi}\label{Def:MVS}
		A Young measure $\nu\in L^{\infty}_{\operatorname{w}}\left([0,T]\times \Omega,\mathcal{M}^1\left(\R^+\times \R^2\right)\right)$ is a \textit{measure-valued solution} to (\ref{eq:CEoriginal}) with initial data $(\rho_0,u_0)$ if $(\rho_0,\rho_0u_0)\in L^1\left(\Omega,\R^+\times \R^2\right)$ and
		\begin{align*}
			\int_{0}^{T}\int_{\Omega}^{}\partial_t\psi \overline{\rho}+\nabla_x\psi \cdot\overline{\rho u}\dxdt + \int_{\Omega}^{}\psi(0,x)\rho_0(x)\dx&=0,\\
			\int_{0}^{T}\int_{\Omega}^{}\partial_t\varphi\cdot\overline{\rho u}+\nabla_x\varphi:\overline{\rho u\otimes u}+\operatorname{div}_x\varphi\overline{p(\rho)} \dxdt+\int_{\Omega}^{}\varphi(0,x)\cdot\rho_0(x)u_0(x)\dx&=0
		\end{align*}
		for all $\psi\in C_c^{\infty}([0,\T)\times\Omega)$ and $\varphi\in C_c^{\infty}\left([0,T)\times\Omega,\R^3 \right)$. Here, we also introduced the notation $\xi=(\xi_1,\xi')\in\R^+\times \R^2$ and
		\begin{align*}
			\overline{\rho}(t,x):=&\int_{\R^+\times \R^2}\xi_1\dd\nu_{(t,x)}(\xi),\\
			\overline{\rho u}(t,x):=&\int_{\R^+\times \R^2}\sqrt{\xi_1}\xi'\dd\nu_{(t,x)}(\xi),\\
			\overline{\rho u\otimes u}(t,x):=&\int_{\R^+\times \R^2}\xi'\otimes \xi' \dd\nu_{(t,x)}(\xi),\\
			\overline{p(\rho)}(t,x):=&\int_{\R^+\times \R^2}p(\xi_1)\dd\nu_{(t,x)}(\xi).
		\end{align*}
		Again, the above integrals have to exist as part of the definition.\\
		Moreover, we say that a sequence $(\rho_n,u_n)$ of weak solutions to (\ref{eq:CEoriginal}) generates the measure-valued solution $\nu$ if $(\rho_n,\sqrt{\rho_n}u_n)\in L^1\left([0,T]\times \Omega,\R^+\times \R^2\right)$ and $(\rho_n,\sqrt{\rho_n}u_n)$ generates $\nu$ in the sense of Young measures.
	\end{defi}
	The latter type of solution still might not take into account possible effects of concentration, which occur for example for weakly converging sequences in the nonlinear terms of (\ref{eq:CEoriginal}). So, we introduce a generalized version of measure-valued solutions that remembers such effects, following~\cite{GSW15}.
	\begin{defi}\label{Def:gMVS}
		A generalized Young measure
		\begin{align*}
			(\nu,m,\nu^{\infty})\in L^{\infty}_{\operatorname{w}}\left([0,T]\times \Omega,\mathcal{M}^1\left(\R^+\times\R^2\right)\right)\times \mathcal{M}^+([0,T]\times \bar{\Omega})\times L^{\infty}_{\operatorname{w}}\left([0,T]\times \bar{\Omega},m,\mathcal{M}^1\left(\left(\mathbb{S}^{2}_{\gamma,2}\right)^+\right)\right)
		\end{align*}
		is a \textit{generalized measure-valued solution} to (\ref{eq:CEoriginal}) with initial data $(\rho_0,u_0)$ and pressure $p(\rho)=\rho^{\gamma}$ for $\gamma>1$ if $(\rho_0,\rho_0u_0)\in L^1\left(\Omega,\R^+\times \R^2\right)$ and
		\begin{align*}
			\int_{0}^{T}\int_{\Omega}^{}\partial_t\psi\overline{\rho}+\nabla_x\psi \cdot\overline{\rho u}\dxdt + \int_{\Omega}^{}\psi(0,x)\rho_0(x)\dx&=0,\\
			\int_{0}^{T}\int_{\Omega}^{}\partial_t\varphi\cdot\overline{\rho u}+\nabla_x\varphi:\overline{\rho u\otimes u}+\operatorname{div}_x\varphi\overline{\rho^{\gamma}} \dxdt+\int_{\Omega}^{}\varphi(0,x)\cdot\rho_0(x)u_0(x)\dx&=0
		\end{align*}
		for all $\psi\in C_c^{\infty}([0,\T)\times\Omega)$ and $\varphi\in C_c^{\infty}\left([0,T)\times\Omega,\R^3 \right)$. Here, we introduced the notation $\left(\mathbb{S}^{l+m-1}_{p,q}\right)^+:=\left\{(\beta_1,\beta')\in\mathbb{S}_{p,q}^{l+m-1}\,:\,\beta_1\geq 0 \right\}$ and $\lambda=(\lambda_1,\lambda')\in\R^+\times\R^2$ as well as $\beta=(\beta_1,\beta')\in\left(\mathbb{S}^{2}_{\gamma,2}\right)^+$ together with
		\begin{align*}
			\overline{\rho}(t,x):=&\int_{\R^+\times \R^2}\lambda_1\dd\nu_{(t,x)}(\lambda),\\
			\overline{\rho u}(t,x):=&\int_{\R^+\times \R^2}\sqrt{\lambda_1}\lambda'\dd\nu_{(t,x)}(\lambda),\\
			\overline{\rho u\otimes u}(t,x):=&\int_{\R^+\times \R^2}\lambda'\otimes \lambda' \dd\nu_{(t,x)}(\lambda)+\int_{\R^+\times \R^2}\beta'\otimes \beta' \dd\nu^{\infty}_{(t,x)}(\beta)m,\\
			\overline{\rho^{\gamma}}(t,x):=&\int_{\R^+\times \R^2}\lambda_1^{\gamma}\dd\nu_{(t,x)}(\lambda)+\int_{\R^+\times \R^2}\beta_1^{\gamma}\dd\nu^{\infty}_{(t,x)}(\beta)m,\\
			%\overline{\rho |v|^2}(t,x):=&\int_{\R^+\times \R^2}|\lambda'|^2\dd\nu_{(t,x)}(\lambda)+\int_{\R^+\times\R^2}|\beta'|^2\dd\nu^{\infty}_{(t,x)}(\beta)m.
		\end{align*}
		where we interpret $m\,\dxdt=\dd m(t,x)$. Again, the above integrals have to exist as part of the definition.\\
		Moreover, we say that a sequence $(\rho_n,u_n)$ of finite energy weak solutions to (\ref{eq:CEoriginal}) generates the generalized measure-valued solution $(\nu,m,\nu^{\infty})$ if $(\rho_n,\sqrt{\rho_n}u_n)\in L^{\gamma}\left([0,T]\times \Omega,\R^+\right)\times L^2\left([0,T]\times \Omega,\R^2\right)$ and $(\rho_n,\sqrt{\rho_n}u_n)$ generates $(\nu,m,\nu^{\infty})$ in the sense of generalized Young measures with respect to the function space $\mathcal{F}_{\gamma,2}$.
	\end{defi}
	It will be clear from the context in what sense the overlined variables have to be understood.\\
	Note that the $\gamma$-$2$-recession functions of $(\beta\mapsto \beta_1)$ and $\left(\beta\mapsto\sqrt{\beta_1}\beta'\right)$ are zero. Hence, we left out the corresponding concentration terms in the above definition.
	\begin{rem}
		We want to clarify here that a (generalized) measure-valued solution to the compressible Euler system is not necessarily by definition generated by a sequence of weak solutions. In fact, the purpose of the current contribution is to show the opposite.
	\end{rem}
	Let us also give a definition of what we mean by a vanishing viscosity limit coming from the compressible Navier-Stokes equations.
	\begin{defi}\label{Def:vanishingviscosity}
		We say a sequence $(\rho_n,u_n)\subset L^1\left((0,T)\times \Omega,\R^+\times \R^2\right)$ is a \textit{vanishing viscosity sequence} if $(\rho_n,u_n)$ is a distributional solution of the compressible Navier-Stokes system
		\begin{align}
			\begin{split}
				\partial_t(\rho_nu_n)+\operatorname{div}(\rho_n u_n\otimes u_n)+\nabla \rho_n^{\gamma}&=\mu_n\operatorname{div}\mathbb{S}(\nabla u_n),\\ \label{eq:CNSE}
				\partial_t \rho_n+\operatorname{div}(\rho_n u_n)&=0
			\end{split}
		\end{align}
		with initial data $(\rho_{n,0},u_{n,0})$ satisfying the energy inequality for all $n\in \N$. Here, $(\mu_n)\subset \R^+$ is a null sequence, $(\rho_{n,0})$ converges weakly in $L^{\gamma}\left(\Omega,\R^+\right)$, $\underset{n\in\N}{\inf}\rho_{n,0}\geq c>0$, and $\sqrt{\rho_{n,0}}u_{n,0}$ converges weakly in $L^2\left(\Omega,\R^2\right)$. We also introduced the viscosity stress tensor $\mathbb{S}(\nabla u):=\eta\left(\nabla u+\nabla^T u\right)+\lambda (\operatorname{div}u)\mathbb{E}_d$ with $\eta >0$ and $\lambda\geq -\eta$.\\
		We say a vanishing viscosity sequence $(\rho_n,u_n)$ generates a measure-valued solution $\nu$ if $\left(\rho_n,\sqrt{\rho_n}u_n\right)$ generates the Young measure $\nu$. Moreover, we say a vanishing viscosity sequence $(\rho_n,u_n)$ generates a generalized measure-valued solution $(\nu,\mu,\nu^{\infty})$ if $\left(\rho_n,\sqrt{\rho_n}u_n\right)$ generates the generalized Young measure $(\nu,\mu,\nu^{\infty})$ with respect to the function space $\mathcal{F}_{\gamma,2}$.
	\end{defi}
\begin{rem}\label{Rem:vanishingviscosity}
	Note that for example by~\cite{FNP} for every initial data $(\rho_{n,0}, u_{n,0})$ there exists a weak solution to (\ref{eq:CNSE}) satisfying the energy inequality. Thus, a posteriori it holds that $\left(\rho_n,\sqrt{\rho_n}u_n\right)\subset L^{\gamma}\left([0,T]\times \Omega,\R^+\right)\times L^2\left([0,T]\times \Omega,\R^2\right)$.
\end{rem}
	Now rewrite (\ref{eq:CEoriginal}) via the substitution $m:=\rho u$, $U:=\frac{m\otimes m}{\rho}-\frac{|m|^2}{2\rho}\mathbb{E}_2$, $q:=p(\rho)+\frac{|m|^2}{2\rho}$ into a linear system of PDEs
	\begin{align}
		\begin{split}
			\partial_t\rho+\operatorname{div}m&=0,\\
			\partial_tm+\operatorname{div}U+\nabla q&=0\label{eq:CEsubsolution}.
		\end{split}
	\end{align}
	Note that $\operatorname{tr}U=\frac{|m|^2}{\rho}-\frac{|m|^2}{2 \rho}2=0$ and $U$ is symmetric. The equation (\ref{eq:CEsubsolution}) is treated as a linear system of first order PDEs with unknowns $(\rho,m,U,q)\colon (0,T)\times \Omega\to \R^+\times \R^2\times S^2_0\times \R^+$.\\
	Henceforth, we confine ourselves to $\Omega=\R^2$ and the case $p(\rho)=\rho^2$,\ i.e.\ we consider the Euler system
	\begin{align}
		\begin{split}
			\partial_t\rho+\operatorname{div}(\rho u)&=0,\\
			\partial_t(\rho u)+\operatorname{div}(\rho u\otimes u)+\nabla \rho^2&=0\label{eq:CEpis2}
		\end{split}
	\end{align}
	on the domain $[0,\infty)\times \R^2$ (or also on $[-T,\infty)\times \R^2$ for $T>0$).\\
	For the isentropic Euler system on $\R^2$ we define a property of solutions which corresponds to the entropy inequality in the sense of hyperbolic conservation laws and usually appears in the context of weak-strong uniqueness. We follow~\cite{CDK} and~\cite{BDS}.
	\begin{defi}
		A bounded weak solution $(\rho,u)\in L^{\infty}\left((0,T)\times \R^2, \R^+\times\R^2\right)$ of (\ref{eq:CEpis2}) with initial data $(\rho_0,u_0)$ is called \textit{admissible} if
		\begin{align*}
			\partial_t\left(\rho\frac{|u|^2}{2}+\rho^2 \right)+\operatorname{div}\left(\left(\rho\frac{|u|^2}{2}+2\rho^2 \right)u \right)\leq &0
		\end{align*}
		holds in the sense of distributions,\ i.e.\ tested against functions in $C_c^{\infty}\left([0,T)\times \R^2,\R^3\right)$.\\
		\\
		We say a (generalized) measure-valued solution $\nu$ (respectively $(\nu,m,\nu^{\infty})$) on $[0,T]\times \R^2$ is \textit{admissible} if
		\begin{align*}
			\partial_t\left(\frac{1}{2}\overline{\rho|u|^2}+\overline{\rho^2}\right)+\operatorname{div}\left(\left(\frac{1}{2}\overline{\rho|u|^2}+2\overline{\rho^2} \right)u \right)\leq &0
		\end{align*}
		holds in the sense of distributions. Here we also introduced
		\begin{align*}
			\overline{\rho |u|^2}(t,x):=&\int_{\R^+\times \R^2}|\xi'|^2 \dd\nu_{(t,x)}(\xi), \text{ respectively},\\
			\overline{\rho |u|^2}(t,x):=&\int_{\R^+\times \R^2}|\lambda'|^2 \dd\nu_{(t,x)}(\lambda)+\int_{\R^+\times \R^2}|\beta'|^2\dd\nu^{\infty}(\beta)m.
		\end{align*}
		Depending on the context the correct meaning of the overlined terms will be clear.
	\end{defi}	
	Our next task is to rewrite (\ref{eq:CEsubsolution}) in the form $\mathcal{A}z=0$, where $\mathcal{A}=\sum_{l=1}^{d}A^{l}\partial_{l}$. To that end, we introduce the state vector $z:=(\rho,m,U,q)$ for the unknowns of equation (\ref{eq:CEsubsolution}) and the homogeneous linear differential operator $\mathcal{A}_L$. Note that $U_{21}=U_{12}$ and $U_{22}=-U_{11}$ due to symmetry and tracelessness of $U$. We set
	\begin{align*}
		A_L^t:=A^{1}:=&\begin{pmatrix}
		1&0&0&0&0&0&0&0\\
		0&1&0&0&0&0&0&0\\
		0&0&1&0&0&0&0&0
		\end{pmatrix},\\
		A_L^{x_1}:=A^{2}:=&\begin{pmatrix}
		0&1&0&0&0&0&0&0\\
		0&0&0&1&0&0&0&1\\
		0&0&0&0&1&0&0&0
		\end{pmatrix},\\
		A_L^{x_2}:=A^{3}:=&\begin{pmatrix}
		0&0&1&0&0&0&0&0\\
		0&0&0&0&1&0&0&0\\
		0&0&0&-1&0&0&0&1
		\end{pmatrix}.
	\end{align*}
	Note that this defines a constant rank operator $\mathcal{A}_L$, cf.\ Lemma~1 in~\cite{CFKW}. Moreover, we have $k=8\geq 3=m$.\\
	Now for a state vector $z=(\rho,m,U,q)$ the corresponding capital letter reads
	\begin{align*}
		Z=\left(\sum_{j=1}^{8}{A_L}^{l}_{ij}z_j\right)_{il}=\begin{pmatrix}
		\rho&m_1&m_2\\
		m_1&U_{11}+q&U_{12}\\
		m_2&U_{12}&-U_{11}+q
		\end{pmatrix}.
	\end{align*}
	Further, let us introduce the lift map $Q$ by
	\begin{align*}
		Q_{\gamma}\colon \R^+\times \R^2&\to\R^+\times \R^2\times S^2_0\times \R^+\subset \R^8,\\
		(\xi_1,\xi')&\mapsto\left(\xi_1,\sqrt{\xi_1}\xi',\xi'\otimes\xi'-\frac{|\xi'|^2}{2}\mathbb{E}_2,\xi_1^{\gamma}+\frac{|\xi'|^2}{2}\right). 
	\end{align*}
	It can be checked right away that $Q_{\gamma}\left(\rho,\sqrt{\rho}u\right)$ is a solution to (\ref{eq:CEsubsolution}) in the sense of distributions if $(\rho, u)$ is a weak solution to (\ref{eq:CEoriginal}). We set $Q:=Q_2$ as our case of interest is $\gamma=2$.\\
	We are now ready to prove the main step in the construction of our non-generable measure-valued solutions.
	\begin{prop}\label{Prop:Chiodaroli}
		Let $T>0$. There exists Lipschitz initial data $(\rho_{-T},u_{-T})\colon \R^2\to \R^+\times \R^2$ which gives rise to infinitely many pairs of  admissible weak solutions $(\rho,u),(\tilde{\rho},\tilde{u})\colon (-T,\infty)\times \R^2\to\R^+\times \R^2$ to (\ref{eq:CEpis2}) with $\inf\rho,\inf\tilde{\rho}>0$ and $(\rho,u),(\tilde{\rho},\tilde{u})\in L^{\infty}\left((-T,\infty),\R^+\times \R^2\right)$. These weak solutions are locally Lipschitz and coincide up to time $0$. Moreover, for each such pair of weak solutions there exists an open set $A\subset (0,\infty)\times\R^2$ such that for the corresponding lifted states $z:=Q\left(\rho,\sqrt{\rho}u\right)$ and $\tilde{z}:=Q\left(\tilde{\rho},\sqrt{\tilde{\rho}}\tilde{u}\right)$ holds $z(t,x)-\tilde{z}(t,x)\neq 0$ and $z(t,x)-\tilde{z}(t,x)\notin \Lambda_{\mathcal{A}_L}$ for\ a.e.\ $(t,x)\in A$.
	\end{prop}
	\begin{proof}
		By Corollary~\ref{lemm:charwavecone} it suffices to show that $z(t,x)-\tilde{z}(t,x)\neq 0$ and $\det(Z(t,x)-\tilde{Z}(t,x))\neq 0$ for all $(t,x)\in A\subset (0,\infty)\times\R^2$ for some open set $A$. If $\rho,\tilde{\rho}$ are constant and not equal, then $z\neq \tilde{z}$, and the determinant condition reads
		\begin{align*}
			&\det(Z-\tilde{Z})\\
			=&\left|\begin{pmatrix}
			\rho-\tilde{\rho}&\rho u_1-\tilde{\rho}\tilde{u}_1&\rho u_2-\tilde{\rho}\tilde{u}_2\\
			\rho u_1-\tilde{\rho}\tilde{u}_1&\rho u_1^2+\rho^2-\tilde{\rho}\tilde{u}_1^2-\tilde{\rho}^2&\rho u_1u_2-\tilde{\rho}\tilde{u}_1\tilde{u}_2\\
			\rho u_2-\tilde{\rho}\tilde{u}_2&\rho u_1u_2-\tilde{\rho}\tilde{u}_1\tilde{u}_2&\underbrace{-\rho u_1^2+\rho^2+|u|^2\rho+\tilde{\rho}\tilde{u}_1^2-\tilde{\rho}^2-|\tilde{u}|^2\tilde{\rho}}_{=\rho u_2^2-\tilde{\rho}\tilde{u}_2^2+\rho^2-\tilde{\rho}^2}
			\end{pmatrix} \right|\\
			=&\left(\rho^2-\tilde{\rho}^2\right)\left(-\rho\tilde{\rho}\left((u_1-\tilde{u}_1)^2+(u_2-\tilde{u}_2)^2\right)+\left(\rho^2-\tilde{\rho}^2\right)(\rho-\tilde{\rho})\right)\overset{!}{=}0,
		\end{align*}
		which follows by a tedious computation. Since we assumed $\rho\neq \tilde{\rho}$ to be positive constants, the above is equivalent to the condition
		\begin{align}
			|u-\tilde{u}|^2=(u_1-\tilde{u}_1)^2+(u_2-\tilde{u}_2)^2\overset{!}{=}\frac{(\rho+\tilde{\rho})(\rho-\tilde{\rho})^2}{\rho\tilde{\rho}}\label{eq:weirdwaveconecond}.
		\end{align}
		So, we seek for two weak solutions $(\rho,u),(\tilde{\rho},\tilde{u})$ with positive constant densities $\rho\neq \tilde{\rho}$ which fulfill (\ref{eq:weirdwaveconecond}) almost nowhere on an open set $A$.\\
		The construction of these two solutions will be provided by~\cite{CDK}. In particular, we will choose a shock as initial datum $(\rho_0,u_0)$ which possesses two admissible fan subsolutions (in the sense of~\cite{CDK}). One can then infer admissible weak solutions to (\ref{eq:CEpis2}) with special properties from these two subsolutions. In detail:\\
		Choose
		\begin{align*}
			(\rho_0,u_0):=\begin{cases}
			(\rho_-,u_-)\,,\,x_2<0,\\
			(\rho_+,u_+)\,,\,x_2>0
			\end{cases}
		\end{align*}
		with
		\begin{align*}
			\rho_-:=&1,\\
			\rho_+:=&4,\\
			u_-:=&\left(-\frac{1}{\rho_+},2\sqrt{2}\left(\sqrt{\rho_+}-\sqrt{\rho_-}\right)\right)=\left(-\frac{1}{4},2\sqrt{2}\right),\\
			u_+:=&\left(-\frac{1}{\rho_+},0\right)=\left(-\frac{1}{4},0\right).
		\end{align*}
		Then, by Lemma~6.1 in~\cite{CDK} the shock $(\rho_0,u_0)$ arises from a classical solution with Lipschitz initial data in the sense that there exists $(\rho^c,u^c)\in W^{1,\infty}_{\operatorname{loc}}\cap L^{\infty}\left((-\infty,0)\times\R^2,\R^+\times\R^2\right)$ such that $(\rho^c,u^c)$ solves (\ref{eq:CEpis2}) in the classical sense and $(\rho^c(t,\cdot),u^c(t,\cdot))\overset{t\rightarrow 0^-}{\rightarrow} (\rho_0,u_0)$ almost everywhere. Set $(\rho_{-T},u_{-T}):=(\rho^c(-T,\cdot),u^c(-T,\cdot))$.\\
		Let us now choose the two subsolutions. This means, we choose a triple $(\bar{\rho},\bar{u},\bar{w})\colon (0,\infty)\times \R^2\to \R^+\times \R^2\times S^{2\times 2}_0$ of piecewise constant functions. Our particular choices will be of the form
		\begin{align*}
			(\bar{\rho},\bar{u},\bar{w})=&\left(\rho_-,u_-,u_-\otimes u_--\frac{1}{2}|u_-|^2\mathbb{E}_2\right)\mathds{1}_{P_-}+(\rho_1,u_1,w_1)\mathds{1}_{P_1}\\
			&+\left(\rho_+,u_+,u_+\otimes u_+-\frac{1}{2}|u_+|^2\mathbb{E}_2\right)\mathds{1}_{P_+}
		\end{align*}
		with $P_-,P_1,P_+\subset (0,\infty)\times \R^2$ of the form
		\begin{align*}
			P_-=&\{(t,x)\,:\,t>0\textup{ and }x_2<\nu_-t \},\\
			P_1=&\{(t,x)\,:\,t>0\textup{ and }\nu_-t<x_2<\nu_+t \},\\
			P_-=&\{(t,x)\,:\,t>0\textup{ and }x_2>\nu_+t \},
		\end{align*}
		where $\nu_-,\nu_+\in\R$ with $\nu_-<\nu_+$ and
		\begin{align*}
			u_1\otimes u_1-w_1<\frac{C_1}{2}\mathbb{E}_2
		\end{align*}
		for some $C_1>0$. We introduce $\alpha,\beta,\gamma,\delta\in \R$ by
		\begin{align*}
			u_1=&(\alpha,\beta),\\
			w_1=&\begin{pmatrix}
			\gamma&\delta\\
			\delta&-\gamma
			\end{pmatrix}.
		\end{align*}
		Proposition~5.1 in~\cite{CDK} gives the following characterization: The triple $(\bar{\rho},\bar{u},\bar{w})$ is an admissible fan subsolution in the sense of~\cite{CDK} with the specific choice $\rho_+=4$ and $\rho_-=1$ if and only if
		\begin{align}
			\nu_-(1-\rho_1)=&2\sqrt{2}-\rho_1\beta,\label{eq:cdkcondanfang}\\
			\nu_-\left(-\frac{1}{4}-\rho_1\alpha\right)=&-\frac{1}{\sqrt{2}}-\rho_1\delta,\\
			\nu_-(2\sqrt{2}-\rho_1\beta)=&8+\rho_1\gamma+1-\rho_1^2-\rho_1\frac{C_1}{2},\\
			\nu_+(\rho_1-4)=&\rho_1\beta-0,\\
			\nu_+(\rho_1\alpha-(-1))=&\rho_1\delta-0,\\
			\nu_+(\rho_1\beta-0)=&-\rho_1\gamma-0+\rho_1^2-16+\rho_1\frac{C_1}{2},\\
			\alpha^2+\beta^2<&C_1,\\
			0<&\left(\frac{C_1}{2}-\alpha^2+\gamma\right)\left(\frac{C_1}{2}-\beta^2-\gamma\right)-(\delta-\alpha\beta)^2,\\
			\nu_-\left(1-\rho_1^2\right)+\nu_-\left(\frac{\left(\frac{1}{16}+8\right)}{2}-\rho_1\frac{C_1}{2}\right)\leq &(1+1)2\sqrt{2}-\left(\rho_1^2+\rho_1^2\right)\beta+\sqrt{2}\left(\frac{1}{16}+8\right)-\rho_1\beta\frac{C_1}{2},\\
			\nu_+\left(\rho_1^2-16\right)+\nu_+\left(\rho_1\frac{C_1}{2}-\frac{1}{8}\right)\leq &\left(\rho_1^2+\rho_1^2\right)\beta-0+\rho_1\beta\frac{C_1}{2}-0,\label{eq:cdkcondende}
		\end{align}
		where we used $p(\rho)=\rho^2$,\ i.e.\ $\varepsilon(\rho)=\rho$.\\
		In the proof of Lemma~6.2 of~\cite{CDK} they give an explicit choice of numbers satisfying these conditions, namely
		\begin{align*}
			\beta:=&0,\\
			\delta:=&0,\\
			\nu_+:=&0,
		\end{align*}
		which yields the constraints
		\begin{align*}
			\alpha=&-\frac{1}{4},\\
			\nu_-=&-\frac{7}{2\sqrt{2}},\\
			\rho_1=&\frac{15}{7},\\
			\frac{9049}{1680}<&C_1\leq\frac{11497}{1680},\\
			\gamma=&\frac{C_1}{2}-\frac{559}{105}.
		\end{align*}
		To come up with another admissible fan subsolution we perturb the above choice from~\cite{CDK} slightly. To be precise, we start with defining
		\begin{align*}
			\tilde{\alpha}=&-\frac{1}{4},\\
			\tilde{\nu}_+:=&\eta
		\end{align*}
		for some $\eta<0$. This yields,
		\begin{align*}
			\tilde{\rho}_1=&\frac{15+16\sqrt{2}\eta+12\eta^2}{7+4\sqrt{2}\eta+3\eta^2},\\
			\tilde{\beta}=&\frac{\eta(\tilde{\rho}_1-4)}{\tilde{\rho}_1},\\
			\tilde{\delta}=&-\frac{1}{4}\tilde{\beta},\\
			\tilde{\nu}_-=&-\frac{14\sqrt{2}+29\eta+6\sqrt{2}\eta^2}{(3\eta+2\sqrt{2})^2},\\
			\tilde{\gamma}=&\tilde{\rho}_1-\frac{16}{\tilde{\rho}_1}+\frac{\tilde{C}_1}{2}-\eta\tilde{\beta}.
		\end{align*}
		For all $\eta\in (\mu,0)$ there exists $\tilde{C}_1$ such that the inequalities (4.11)-(4.14) can be satisfied. Here, $\mu\approx-0.2249$ is determined as the first negative value at which the dominating inequalities (4.12) and (4.13) contradict each other. If $|\eta|$ is sufficiently small, then we can choose some $\tilde{C}_1$ in approximately the same interval as for $C_1$.\\
		Now we use the fact that $C_1$ and $\tilde{C}_1$ can be chosen within some range. In particular, to conclude the proof, it suffices to choose $\eta\in\left(\mu,0\right)$ and $C_1,\tilde{C}_1$ such that
		\begin{align}
			\sqrt{\frac{(\rho_1+\tilde{\rho}_1)(\rho_1-\tilde{\rho}_1)^2}{\rho_1\tilde{\rho}_1}}<\left|\sqrt{C_1}-\sqrt{\tilde{C}_1}\right|.\label{eq:ccond}
		\end{align}
		Indeed, for the subsolution $(\bar{\rho},\bar{u},\bar{w})$ we can infer (infinitely many) admissible weak solutions $(\rho,u)$ to (\ref{eq:CEpis2}) by Proposition~3.6 in~\cite{CDK} with $\rho=\bar{\rho}$. This is done with the help of Lemma~3.7 in~\cite{CDK}, which gives us
		\begin{align*}
			|u|^2=\operatorname{tr}(u\otimes u)=\operatorname{tr}(U)+\frac{C_1}{2}\operatorname{tr}(\mathbb{E}_2)=C_1
		\end{align*}
		a.e.\ on $P_1$ for some $S^{2\times 2}_0$-valued function $U$. Analogously, for the subsolution $(\bar{\tilde{\rho}},\bar{\tilde{u}},\bar{\tilde{w}})$ we obtain (infinitely many) admissible weak solutions $(\tilde{\rho},\tilde{u})$ with $\tilde{\rho}=\bar{\tilde{\rho}}$ and
		\begin{align*}
			|\tilde{u}|^2=\tilde{C}_1
		\end{align*}
		a.e.\ on $\tilde{P}_1$. Note that for all $\eta\in\left(\mu,0\right)$ we have $\tilde{\nu}_-<-\frac{7}{2\sqrt{2}}$. Hence, on
		\begin{align*}
			P_1\cap \tilde{P}_1=&\{(t,x)\,:\,t>0\textup{ and }\nu_-t<x_2<\tilde{\nu}_+t \}\\
			=&\left\{(t,x)\,:\,t>0\textup{ and }-\frac{7}{2\sqrt{2}}t<x_2<\eta t \right\}
		\end{align*}
		we obtain by (\ref{eq:ccond}) that
		\begin{align*}
			|u-\tilde{u}|^2\geq \big||u|-|\tilde{u}|\big|^2=\left|\sqrt{C_1}-\sqrt{\tilde{C}_1}\right|^2>\frac{(\rho_1+\tilde{\rho}_1)(\rho_1-\tilde{\rho}_1)^2}{\rho_1\tilde{\rho}_1}
		\end{align*}
		holds almost everywhere. Thus, (\ref{eq:weirdwaveconecond}) is violated\ a.e.\ on $A:=P_1\cap\tilde{P}_1$. Note that $\nu_-=-\frac{7}{2\sqrt{2}}<\mu$, hence $P_1\cap \tilde{P}_1$ is open for all possible choices of $\eta$. It also follows that $\tilde{\nu}_-<\tilde{\nu}_+$.\\
		Therefore, it remains to give a specific choice of $\left(\mu,0\right)\ni \eta$, $C_1$, and $\tilde{C}_1$ such that (\ref{eq:ccond}) holds. For example, choose
		\begin{align*}
			\eta=&-0.001,\\
			C_1=&6,\\
			\tilde{C}_1=&5.8.
		\end{align*}
		One can check that this choice indeed fulfills the required conditions (\ref{eq:cdkcondanfang})-(\ref{eq:cdkcondende}) and (\ref{eq:ccond}). This finishes the proof.\\
		Note that the above choices of $\eta$, $C_1$, and $\tilde{C}_1$ can be made within some interval. Thus, we obtain infinitely many such choices of subsolutions.
	\end{proof}
	\begin{rem}
		For two weak solutions $(\rho,u)$ and $(\rho, \tilde{u})$ with equal density, the corresponding lifted states are always wave-cone-connected. Indeed, we have
		\begin{align*}
			&Z(t,x)-\tilde{Z}(t,x)\\
			=&\begin{pmatrix}
			\rho&\rho u_1&\rho u_2\\
			\rho u_1&\rho u_1u_1-\frac{|u|^2}{2}\rho+\rho^2+\frac{|u|^2}{2}\rho&\rho u_1u_2\\
			\rho u_2&\rho u_1u_2&-\left(\rho u_1u_1-\frac{|u|^2}{2}\rho\right)+\rho^2+\frac{|u|^2}{2}\rho
			\end{pmatrix}(t,x)-\tilde{Z}(t,x)\\
			=&\rho\begin{pmatrix}
			0& (u_1-\tilde{u}_1)& (u_2-\tilde{u}_2)\\
			(u_1-\tilde{u}_1)& \left(u_1^2-\tilde{u}_1^2\right)& (u_1u_2-\tilde{u}_1\tilde{u}_2)\\
			(u_2-\tilde{u}_2)& (u_1u_2-\tilde{u}_1\tilde{u}_2)&\underbrace{- \left(u_1^2-\tilde{u}_1^2\right)+|u|^2-|\tilde{u}|^2}_{=u_2^2-\tilde{u}_2^2}
			\end{pmatrix}(t,x)
		\end{align*}
		Thus, $\det\left(Z(t,x)-\tilde{Z}(t,x)\right)=0$ trivially holds, which is an elementary computation. Note that this holds in particular for $|u|\neq|\tilde{u}|$. 
	\end{rem}	
	We now combine the pairs of states corresponding to the above compression wave with our previous rigidity result Theorem~\ref{Thm:rigidity}.
	\begin{theo}\label{Thm:classicalYM}
		There exists Lipschitz initial data $(\rho_{-T},u_{-T})\colon \R^2\to \R^+\times \R^2$ which gives rise to infinitely many admissible measure-valued solutions to the isentropic Euler system (\ref{eq:CEpis2}) of the form $\nu=\lambda \delta_{\left(\rho,\sqrt{\rho}u\right)}+(1-\lambda)\delta_{\left(\tilde{\rho},\sqrt{\tilde{\rho}}\tilde{u}\right)}$, where $(\rho,u),(\tilde{\rho},\tilde{u})$ are as in Proposition~\ref{Prop:Chiodaroli} with corresponding open set $A=\left\{(t,x)\,:\,t>0\textup{ and }-\frac{7}{2\sqrt{2}}t<x_2<-0.001 t \right\}$ on which the lifted states are not wave-cone-connected, and $\lambda\in(0,1)$. Such a solution $\nu$ coincides with a classical compression wave up to time $0$ and moreover $\nu$ cannot be generated by sequences of weak solutions or vanishing viscosity sequences $(\rho_n,u_n)$ satisfying the following property:\\
		There exists an open and bounded set $B\subset A$ such that the sequences $\left(\rho_n|u_n|^2\right)$ and $\left(\rho_n^2\right)$ are equi-integrable on $B$.
	\end{theo}	
	\begin{rem}
		Note that the sequences $\left(\rho_n|u_n|^2\right)$ and $\left(\rho_n^2\right)$ correspond to the integrands of the energy $E_{(\rho,u)}$.
	\end{rem}
	\begin{proof}
		By Proposition~\ref{Prop:Chiodaroli} there exists a Lipschitz initial datum $(\rho_{-T},u_{-T})\colon \R^2\to \R^+\times \R^2$ which gives rise to a pair of weak solutions $(\rho,u),(\tilde{\rho},\tilde{u})\colon (-T,\infty)\times \R^2\to\R^+\times \R^2$ with $\inf\rho>0$, $\inf \tilde{\rho}>0$ and $(\rho,u),(\rho,\tilde{u})\in L^{\infty}\left((-T,\infty)\times \R^2,\R^+\times \R^2\right)$. These solutions all coincide with a classical compression wave up to time $0$. Moreover, this proposition ensures that there is an open set $A\subset (0,\infty)\times\R^2$ such that for the corresponding lifted states $z=Q\left(\rho,\sqrt{\rho}u\right)$ and $\tilde{z}=Q\left(\tilde{\rho},\sqrt{\tilde{\rho}}\tilde{u}\right)$ holds $z(t,x)-\tilde{z}(t,x)\neq 0$ and $z(t,x)-\tilde{z}(t,x)\notin \Lambda_{\mathcal{A}_L}$ for\ a.e.\ $(t,x)\in A$. Note that $z,\tilde{z}\in L^{\infty}\left((-T,\infty)\times \R^2,\R^{8}\right)$. The map $Q$ is clearly continuous. Since $Q\left(\rho,\sqrt{\rho}u\right)\rightarrow 0$ as $\rho\rightarrow 0^+$ for all $v\in\R^2$, we can extend $Q$ by zero yielding a continuous function $\tilde{Q}\colon \R^3\to \R^8$.\\
		Now define the Young measures 
		\begin{align*}
			(t,x)&\mapsto\lambda\delta_{\left(\rho,\sqrt{\rho}u\right)(t,x)}+(1-\lambda)\delta_{\left(\tilde{\rho},\sqrt{\tilde{\rho}}\tilde{u}\right)(t,x)}=:\nu_{(t,x)},\\
			(t,x)&\mapsto\lambda\delta_{z(t,x)}+(1-\lambda)\delta_{\tilde{z}(t,x)}=:\tilde{\nu}_{(t,x)}
		\end{align*}
		for $\lambda\in (0,1)$ fixed but arbitrary. Note that for all $f\in C_0\left(\R^8,\R\right)$ and\ a.e.\ $(t,x)\in (-T,\infty)\times \R^2$ holds
		\begin{align*}
			\int_{\R^8}^{}f(z)\dd\tilde{\nu}_{(t,x)}(z)=\int_{\R^+\times\R^2} (f\circ Q)(\xi)\dd\nu_{(t,x)}(\xi).
		\end{align*}
		One immediately checks that $\nu$ is a measure-valued solution to (\ref{eq:CEpis2}).\\
		Assume now that $\nu$ is generated by a sequence $(\rho_n,u_n)\colon (-T,\infty)\times \R^2\to\R^+\times \R^2$ of weak solutions to (\ref{eq:CEpis2}) with the property that there exists an open and bounded set $B\subset A$ such that the sequences $\left(\rho_n|u_n|^2\right)$ and $\left(\rho_n^2\right)$ are equi-integrable on $B$. Then, $\left(Q\circ \left(\rho_n,\sqrt{\rho_n}u_n\right)\right)$ is $\mathcal{A}_L$-free on $(-T,\infty)\times \R^2$ in the sense of distributions. In particular, $\left(Q\circ \left(\rho_n,\sqrt{\rho_n}u_n\right)\big|_B\right)$ is $\mathcal{A}_L$-free in the weak sense also on the open and bounded set $B$.\\
		If $(\rho_n,u_n)$ is a vanishing viscosity sequence generating the measure-valued solution $\nu$ we have to be more careful:\\		
		Let $(\mu_n)\subset \R^+$ be the corresponding null sequence of viscosity parameters and let $(\rho_{n,-T},u_{n,-T})$ be the corresponding initial data as in Definition~\ref{Def:vanishingviscosity}. By our Definition~\ref{Def:vanishingviscosity} $(\rho_n,u_n)$ fulfills the energy inequality
		\begin{align*}
			&\int_{\R^2}^{}\frac{1}{2}\rho_n(t,x)|u_n(t,x)|^2+\frac{1}{2}\rho_n^2(t,x)\dx+\int_{-T}^{t}\int_{\R^2}^{}\mu_n\mathbb{S}(\nabla u_n):\nabla u_n\dxdt\\
			\leq &\int_{\R^2}^{}\frac{1}{2}\rho_{n,-T}(x)|u_{n,-T}(x)|^2+\frac{1}{2}\rho_{n,-T}^2\dx
		\end{align*}
		for\ a.e.\ $t\in (-T,\infty)$. The sequences $\left(\rho_{n,-T}\right)$ and $\left(\sqrt{\rho_{n,-T}}u_{n,-T} \right)$ converge weakly in $L^2$, thus the right hand side of the above inequality is uniformly bounded. Note that by Korn's inequality the term involving the viscosity stress tensor is non-negative.\\
		Since $(\rho_n,u_n)$ is a weak solution to (\ref{eq:CNSE}), its lift satisfies in particular
		\begin{align*}
			\mathcal{A}_L\left(Q\circ\left(\rho_n,\sqrt{\rho_n}u_n\right)\right)=\mu_n \operatorname{div}\mathbb{S}(\nabla u_n)\text{ in }\mathcal{D}'\left(B,\R^3\right).
		\end{align*}
		Hence, we estimate
		\begin{align*}
			\left\|\mathcal{A}_L\left(Q\circ\left(\rho_n,\sqrt{\rho_n}u_n\right)\right)\right\|_{H^{-1}(B,\R^3)}&=\underset{\underset{\varphi\in C_c^{\infty}\left(B,\R^3\right)}{\|\varphi\|_{H^1_0\left(B,\R^3\right)}\leq 1}}{\sup}|\langle \mu_n \operatorname{div}\mathbb{S}(\nabla u_n),\varphi\rangle|\\
			&=\underset{\underset{\varphi\in C_c^{\infty}\left(B,\R^3\right)}{\|\varphi\|_{H^1_0\left(B,\R^3\right)}\leq 1}}{\sup}\left|\int_{B}^{}\mu_n\mathbb{S}(\nabla u_n):\nabla \varphi\dxdt\right|\\
			&\leq \sqrt{\mu_n}\sqrt{\mu_n\int_{B}|\mathbb{S}(\nabla u_n)|^2\dxdt}\cdot 1\\
			&\leq C\sqrt{\mu_n}\sqrt{\mu_n\int_{B}\mathbb{S}(\nabla u_n):\nabla u_n\dxdt}\\
			&\leq C\sqrt{\mu_n}\rightarrow 0
		\end{align*}
		by Korn's inequality and the energy inequality. In particular, we obtain
		\begin{align*}
			\left\|\mathcal{A}_L\left(Q\circ\left(\rho_n,\sqrt{\rho_n}u_n\right)\right)\right\|_{W^{-1,r}\left(B,\R^3\right)}\leq \left\|\mathcal{A}_L\left(Q\circ\left(\rho_n,\sqrt{\rho_n}u_n\right)\right)\right\|_{H^{-1}\left(B,\R^3\right)}\rightarrow 0
		\end{align*}
		for all $r\in \left(1,\frac{3}{2}\right)$.\\		
		Now we treat both cases of $(\rho_n,u_n)$ being a vanishing viscosity sequence or a sequence of weak solutions at once:\\
		For all $f\in C_0\left(\R^8\right)$ we have $f\circ \tilde{Q}\in C_0\left(\R^+\times\R^2,\R\right)$, since $\tilde{Q}\in C\left(\R^+\times\R^2,\R^8\right)$. Thus, the sequence $\left(Q\circ \left(\rho_n,\sqrt{\rho_n}u_n\right)\right)$ generates the Young measure $\tilde{\nu}$ on $(-T,\infty)\times\R^2$. In particular, $\left(Q\circ \left(\rho_n,\sqrt{\rho_n}u_n\right)\big|_B\right)$ generates the Young measure $\tilde{\nu}\big|_B$.\\
		Moreover, the sequence $\left(Q\circ \left(\rho_n,\sqrt{\rho_n}u_n\right)\big|_B\right)$ is equi-integrable on the open and bounded set $B\subset A$, which of course implies its $L^1\left(B,\R^8\right)$-boundedness. Note also that
		\begin{align*}
			\operatorname{supp}(\tilde{\nu}_x)=\{z(x) \}\cup \{\tilde{z}(x) \}\subset \left\{\lambda z(x)+(1-\lambda)\tilde{z}(x)\,:\, \lambda\in[0,1] \right\}
		\end{align*}
		for\ a.e.\ $x\in B$. Therefore, Theorem~\ref{Thm:rigidity} implies that $\tilde{\nu}\big|_B=\delta_{w}$ for some $w\in L^1\left(B,\R^m\right)$. But this contradicts the fact that $\tilde{\nu}_b= \lambda \delta_{z(b)}+(1-\lambda)\delta_{\tilde{z}(b)}$ for\ a.e.\ $b\in B$ with $\lambda\in (0,1)$, which finishes the proof.\\
		The two weak solutions $(\rho, u)$ and $(\tilde{\rho},\tilde{u})$ from Proposition~\ref{Prop:Chiodaroli} are admissible and hence the admissibility is inherited by the measure-valued solution $\nu= \lambda\delta_{\left(\rho,\sqrt{\rho}u\right)}+(1-\lambda)\delta_{\left(\tilde{\rho},\sqrt{\tilde{\rho}}\tilde{u}\right)}$ by an elementary calculation.
	\end{proof}
	Our main theorem is now just an application of the previous results. This essentially follows from the fact that generating a generalized Young measure is a stronger notion than generating a classical Young measure. In particular, the former implies equi-integrability of the generating sequence, as we will see below.
	\begin{theo}\label{Thm:generalizedYM}
		Let $T>0$. There exists Lipschitz initial data $(\rho_{-T},u_{-T})\colon \R^2\to \R^+\times \R^2$ which gives rise to infinitely many admissible generalized measure-valued solutions to the two-dimensional isentropic Euler system (\ref{eq:CEpis2}). Such a generalized measure-valued solution coincides on $[-T,0)$ with a classical compression wave and evolves on $(0,\infty)$ such that it cannot be generated by sequences of finite energy weak solutions or by a vanishing viscosity sequence.
	\end{theo}
	\begin{proof}
		Take the initial data $(\rho_{-T},u_{-T})$ and one of the admissible measure-valued solutions $\nu=\lambda \delta_{\left(\rho,\sqrt{\rho}u\right)}+(1-\lambda)\delta_{\left(\tilde{\rho},\sqrt{\tilde{\rho}}\tilde{u}\right)}$ as in Theorem~\ref{Thm:classicalYM}, where $(\rho,u),(\tilde{\rho},\tilde{u})$ are as in Proposition~\ref{Prop:Chiodaroli}. Write this as a generalized Young measure $(\nu,0,\mu)$, where $\mu$ is just a dummy variable completing our notation, because the concentration-measure $\nu^{\infty}$ is defined only $m$-a.e., which in the present case is the empty set. One immediately checks that $(\nu,0,\mu)$ is an admissible generalized measure-valued solution to (\ref{eq:CEpis2}).\\
		Now assume for the sake of contradiction that the generalized measure-valued solution $(\nu,0,\mu)$ is generated by $(\rho_n,u_n)$ which is a sequence  of weak solutions or a vanishing viscosity sequence. By our Definitions~\ref{Def:gMVS} and~\ref{Def:vanishingviscosity} this means that $\left(\rho_n,\sqrt{\rho_n}u_n\right)\in L^{2}\left((-T,\infty)\times\R^2,\R^+\right)\times L^2\left((-T,\infty)\times \R^2,\R^2\right)$ generates $(\nu,0,\mu)$ in the sense of generalized Young measures with respect to the function space $\mathcal{F}_{2,2}$. Define the maps
		\begin{align*}
			\Phi_1\colon \R^+\times \R^2\to\R,\ (\xi_1,\xi')&\mapsto \xi_1^2,\\
			\Phi_2\colon \R^+\times \R^2\to\R,\ (\xi_1,\xi')&\mapsto |\xi'|^2.
		\end{align*}
		For all $f\in \mathcal{F}_1$ we obtain that
		\begin{align*}
			(f\circ(\operatorname{id}\times \Phi_1))^{\infty}(x,\beta_1,\beta_2)&=\underset{\underset{\underset{s\rightarrow\infty}{(\beta_1',\beta_2')\rightarrow(\beta_1,\beta_2)}}{x'\rightarrow x}}{\lim}\frac{(f\circ(\operatorname{id}\times \Phi_1))\left(x',s^2\beta_1',s^2\beta_2'\right)}{s^{4}}\\
			&=\underset{\underset{\underset{s\rightarrow\infty}{(\beta_1',\beta_2')\rightarrow(\beta_1,\beta_2)}}{x'\rightarrow x}}{\lim}\frac{f\left(x',s^4(\beta_1')^2\right)}{s^{4}}\\
			&=\underset{\underset{\underset{\tilde{s}\rightarrow\infty}{\tilde{\beta}_1'\rightarrow\tilde{\beta}_1}}{x'\rightarrow x}}{\lim}\frac{f\left(x',\tilde{s}\tilde{\beta}_1\right)}{\tilde{s}}
		\end{align*}
		exists and is continuous on $\bar{\Omega}\times \left(\mathbb{S}^{2}_{2,2}\right)^+$. This means that $f\circ (\operatorname{id}\times\Phi_1)\in \mathcal{F}_{2,2}$. Similarly, one checks that $f\circ (\operatorname{id}\times\Phi_2)\in \mathcal{F}_{2,2}$. Now define the generalized Young measures $(\nu_1,0,\mu)$ and $(\nu_2,0,\mu)$ by
		\begin{align*}
			\langle \nu_i,f\rangle :=\langle \nu,f\circ(\operatorname{id}\times \Phi_i)\rangle \text{ for all }f\in\mathcal{F}_1\text{ and }i=1,2.
		\end{align*}
		Hence, the sequences $\left(\rho_n^2\right)\in L^1\left((-T,\infty)\times \R^2,\R^+\right)$ and $\left(\rho_n|u_n|^2\right)\in L^1\left((-T,\infty)\times \R^2,\R^+\right)$ generate the generalized Young measures $(\nu_1,0,\mu)$ and $(\nu_2,0,\mu)$, respectively.\\
		The remarks in Section~2.4 of~\cite{KR}, which characterize the function space $\mathcal{F}_1$, together with Theorem~2.9 in~\cite{AB} yield that a sequence $(z_n)\in L^1$ generating a generalized Young measure $(\nu,0,\mu)$ with respect to the function space $\mathcal{F}_1$ is equi-integrable. Therefore, the sequences $\left(\rho_n^2\right)$ and $\left(\rho_n|u_n|^2\right)$ are equi-integrable on $(-T,\infty)\times \R^2$. By Lemma~\ref{Lemma:generatingmeasures} the sequence $(\rho_n,u_n)$ generates the measure-valued solution $\nu$ in the classical sense. Thus, Theorem~\ref{Thm:classicalYM} yields the desired contradiction.
	\end{proof}
	\begin{rem}\label{Rem:rangeofgamma}
		Due to Remark~6.3 in~\cite{CDK}, the results in \cite{CDK} also hold for $p(\rho)=\rho^{\gamma}$ with $\gamma$ in some neighborhood of $2$. Observe also that (\ref{eq:weirdwaveconecond}) is valid for all $\gamma>1$ if we replace the right-hand side by $\frac{(\rho-\tilde{\rho})(\rho^{\gamma}-\tilde{\rho^{\gamma}})}{\rho\tilde{\rho}}$. So, our construction depends continuously on $\gamma$. Thus, also the conclusions of Theorem~\ref{Thm:generalizedYM} hold in a neighborhood of $2$.
	\end{rem}
	Theorem~\ref{Thm:generalizedYM} provides us with a concrete instance where the following non-generability result for certain generalized measure-valued solutions without concentration part can be applied. This result can be interpreted as a selection principle for unphysical solutions.
	\begin{cor}\label{Cor:selectioncriterion}
		Let $d\in \N$ and $\Omega\subset \R^d$ be open and bounded. Further, let $(\nu,0,\mu)$ be a generalized measure-valued solution to (\ref{eq:CEoriginal}) with pressure law $p(\rho)=\rho^{\gamma}$ for $\gamma>1$ on $[0,T]\times \Omega$ of the form $\nu=\lambda\delta_{\left(\rho,\sqrt{\rho}u\right)}+(1-\lambda)\delta_{\left(\tilde{\rho},\sqrt{\tilde{\rho}}\tilde{u}\right)}$ with $\lambda\in (0,1)$, where $(\rho,u)$ and $(\tilde{\rho},\tilde{u})$ are weak solutions (to possibly different initial data) whose lifts are not wave-cone-connected on a set of positive measure. Then $(\nu,0,\mu)$ cannot be generated by sequences of finite energy weak solutions or by a vanishing viscosity sequence.
	\end{cor}
	\begin{proof}
		This is a direct consequence of Lemma~\ref{Lemma:generatingmeasures}, Theorem~\ref{Thm:rigidity}, and the proofs of Theorems~\ref{Thm:classicalYM} and~\ref{Thm:generalizedYM}. 
	\end{proof}

	\section*{Acknowledgements}
	The authors would like to thank Jack Skipper for many insightful discussions.


\begin{thebibliography}{99}
		%\begin{thebibliography}
		% please try to use the phi system -
		% the references should be in alphabetical order of authors' names.
		% Articles with a single author first, author will 1 co-author next,
		% then author with several co-authors;
		
		
		%\bibitem{BDSV}  T. Buckmaster, C. De Lellis, L. Sz\'ekelyhidi Jr., and V. Vicol,
		%Onsager's conjecture for admissible weak solutions (2017), To appear in {\it Commun. Pure Appl. Math.}
		
		\bibitem{AB} J.J.~Alibert, G.~Bouchitt\'e, Non-uniform integrability and generalized Young measures. {\em J. Conv. Anal.}~{\bf 4} (1997), no.~1, 129--147.
		
		\bibitem{BJ87} J.M.~Ball, R.D.~James, Fine phase mixtures as minimizers of energy. {\em Arch. Rational Mech. Anal.}~{\bf 100} (1987), no.~1, 13--52. 
		
		\bibitem{BTW} C.~Bardos, E.S.~Titi, E.~Wiedemann, The vanishing viscosity as a selection principle for the Euler equations: the case of 3D shear flow. {\em C. R. Math. Acad. Sci. Paris}~{\bf 350} (2012), no.~15--16, 757--760.
		
		\bibitem{BDS} Y.~Brenier, C.~De Lellis, L.~Sz\'ekelyhidi, Jr., Weak-strong uniqueness for measure-valued solutions. {\em Comm. Math. Phys.}~{\bf 305} (2011), no.~2, 351--361.
		
		\bibitem{Chi14} E.~Chiodaroli, A counterexample to well-posedness of entropy solutions to the compressible Euler system. {\em J. Hyperbolic Differ. Equ.}~{\bf 11} (2014), no.~3, 493--519.
		
		\bibitem{CDK} E.~Chiodaroli, C.~De Lellis, O.~Kreml, Global ill-posedness of the isentropic system of gas dynamics. {\em Comm. Pure Appl. Math.}~{\bf 68} (2015), no.~7, 1085--1283.
		
		\bibitem{CFKW} E.~Chiodaroli, E.~Feireisl, O.~Kreml, and E.~Wiedemann, $\mathcal{A}$-free rigidity and applications to the compressible Euler system. {\em Ann. Mat. Pura Appl. (4)}~{\bf 196} (2017), no.~4, 1557--1572.
		
		\bibitem{DS09} C.~De Lellis and L.~Sz\'ekelyhidi, Jr., The Euler equations as a differential inclusion. {\em Ann. of Math. (2)}~{\bf 170} (2009), no.~3, 1417--1436.
		
		\bibitem{DS10} C.~De Lellis and L.~Sz\'ekelyhidi, Jr., On admissibility criteria for weak solutions of the Euler equations. {\em Arch. Ration. Mech. Anal.}~{\bf 195} (2010), no.~1, 225--260.
		
		\bibitem{DPR} G.~De Philippis, L.~Palmieri, F.~Rindler, On the two-state problem for general differential operators. {\em Nonlinear Anal.}~{\bf 177} (2018), part~B, 387--396.
		
		\bibitem{DM87} R.J.~DiPerna, A.J.~Majda, Oscillations and concentrations in weak solutions of the incompressible fluid equations. {\em Comm. Math. Phys.}~{\bf 108} (1987), no.~4, 667--689.
		
		\bibitem{FNP} E.~Feireisl, A.~Novotn\'y, H.~Petzeltov\'a, On the existence of globally defined weak solutions to the Navier-Stokes equations. {\em J. Math. Fluid Mech.}~{\bf 3} (2001), no.~4, 358--392.
		
		\bibitem{FKMT} U.S.~Fjordholm, R.~K\"appeli, S.~Mishra, E.~Tadmor, Construction of approximate entropy measure-valued solutions for hyperbolic systems of conservation laws. {\em Found. Comput. Math.}~{\bf 17} (2017), no.~3, 763--827.
		
		\bibitem{FM} I.~Fonseca, S.~M\"uller, $\mathcal{A}$-quasiconvexity, lower semicontinuity, and Young measures. {\em SIAM J. Math. Anal.}~{\bf 30} (1999), no.~6, 1355--1390.
		
		\bibitem{GSW15} P.~Gwiazda, A.~\'Swierczewska-Gwiazda, E.~Wiedemann, Weak-strong uniqueness for measure-valued solutions of some compressible fluid models. {\em Nonlinearity}~{\bf 28} (2015), no.~11, 3873--3890. 
		
		\bibitem{KR} J.~Kristensen, F.~Rindler, Characterization of generalized gradient Young measures generated by sequences in $W^{1,1}$ and $BV$. {\em Arch. Ration. Mech. Anal.}~{\bf 197} (2010), no.~2, 539--598.
		
		\bibitem{Neu93} J.~Neustupa, Measure-valued solutions of the Euler and Navier-Stokes equations for compressible barotropic fluids. {\em Math. Nachr.}~{\bf 163} (1993), 217--227.
		
		\bibitem{SW12} L.~Sz\'ekelyhidi, Jr., E.~Wiedemann, Young measures generated by ideal incompressible fluid flows. {\em Arch. Ration. Mech. Anal.}~{\bf 206} (2012), no.~1, 333--366.
		
		\bibitem{Wie11} E.~Wiedemann, Existence of weak solutions for the incompressible Euler equations. {\em Ann. Inst. H.
		Poincar\'e Anal. Non Lin\'eaire}~{\bf 28} (2011), no.~5, 727--730.
	\end{thebibliography}
\end{document}